\documentclass[a4paper]{amsart}
\title[{\tiny Stability conditions on threefolds with nef tangent bundles}]{Stability conditions on threefolds with nef tangent bundles}
\date{}
\author{Naoki Koseki}

\makeatletter
 
  \@addtoreset{equation}{section}
 \makeatother

\usepackage{amsmath, amssymb, amsthm, amscd, comment}
\usepackage[frame,cmtip,curve,arrow,matrix,line,graph]{xy}

\theoremstyle{plain}
\newtheorem{thm}{Theorem}[section]
\newtheorem{prop}[thm]{Proposition}
\newtheorem{lem}[thm]{Lemma}
\newtheorem{cor}[thm]{Corollary}
\newtheorem*{thm*}{Theorem}

\theoremstyle{definition}
\newtheorem{defin}[thm]{Definition}
\newtheorem{conj}[thm]{Conjecture}
\newtheorem*{NaC}{Notation and Convention}
\newtheorem*{ACK}{Acknowledgement}

\theoremstyle{remark}
\newtheorem{rmk}[thm]{Remark}

\DeclareMathOperator{\ch}{ch}

\DeclareMathOperator{\Spec}{Spec}
\DeclareMathOperator{\id}{id}
\newcommand{\dR}{\mathbf{R}}
\newcommand{\dL}{\mathbf{L}}
\newcommand{\bP}{\mathbb{P}}
\newcommand{\mcF}{\mathcal{F}}
\newcommand{\mcG}{\mathcal{G}}
\newcommand{\mcO}{\mathcal{O}}
\DeclareMathOperator{\Hom}{Hom}
\DeclareMathOperator{\Coh}{Coh}
\DeclareMathOperator{\rep}{rep}

\DeclareMathOperator{\Pic}{Pic}
\DeclareMathOperator{\NS}{NS}

\DeclareMathOperator{\ext}{ext}
\DeclareMathOperator{\Ext}{Ext}

\DeclareMathOperator{\Sym}{Sym}

\DeclareMathOperator{\Coker}{Coker}

\begin{document}
\maketitle

\begin{abstract}
In this paper, 
we prove the Bogomolov-Gieseker type inequality conjecture 
for threefolds with nef tangent bundles. 
As a corollary, there exist Bridgeland stability conditions 
on these threefolds. 
\end{abstract}

\setcounter{tocdepth}{1}
\tableofcontents


\section{Introduction}

\subsection{Motivation and results}
The construction of Bridgeland stability conditions 
on an algebraic variety $X$ is an important problem. 
When $X$ is a surface, the existence of Bridgeland 
stability conditions on $X$ is proved by 
Bridgeland (cf. \cite{bri08}) and Arcara-Bertram (cf. \cite{ab13}). 
It has found many applications 
to classical problems in algebraic geometry, 
especially in the study of birational geometry 
of the moduli spaces of Gieseker stable sheaves 
(see e.g. \cite{abch13, bm14c, bm14b, bmw14, ch14, ch15, ch16, chw17, lz16, lz18}). 

When $X$ is a threefold, 
the existence of Bridgeland stability condisions 
is an open problem in general. 
In the paper \cite{bmt14a}, Bayer, Macr{\`i}, and Toda 
reduced the problem to the so-called 
Bogomolov-Gieseker (BG) type inequality conjecture. 
The BG type inequality conjecture is known to be true 
for Abelian threefolds (cf. \cite{bms16, mp16a, mp16b}), 
Fano threefolds of Picard rank one (cf. \cite{li15}), 
some toric threefolds (cf. \cite{bmsz17}), 
product threefolds of projective spaces 
and Abelian varieties (cf. \cite{kos17}), 
and quintic threefolds (cf. \cite{li18}). 
However, counter-examples of 
the original BG type inequality conjecture are 
constructed (see e.g. \cite{ms19}). 
The failure of the conjecture is related to 
the existence of a kind of negative effective 
divisors on a threefold 
(see Lemma \ref{neg lem}). 
The modification of the conjecture is discussed 
in the paper \cite{bmsz17}, 
and they prove that the modified version of 
the BG type inequality holds 
when $X$ is a Fano threefold of arbitrary Picard rank. 

On the other hand, we can still expect that 
the original BG type inequality conjecture 
will be true if every effective divisor on $X$ 
satisfies a certain positivity condition, e.g. 
if the pseudo-effective cone agrees with the nef cone. 
Actually, in this paper, 
we prove that the original conjecture is true 
for one class of threefolds satisfying this property, 
namely those with nef tangent bundles: 

\begin{thm}
\label{main thm}
Let $X$ be a smooth projective threefold with 
nef tangent bundle. 
Then the original BG type inequality conjecture 
holds for $X$. 
\end{thm}

In particular, the above theorem implies 
the existence of Bridgeland stability conditions 
on these threefolds: 

\begin{thm}
Let $X$ be as in Theorem \ref{main thm}. 
Then there exist Bridgeland stability conditions on $X$. 
\end{thm}

See Theorem \ref{first thm}, Corollary \ref{first cor} 
and Theorem \ref{second thm} 
for the precise statements.


\subsection{Relation to existing works} 
First recall that threefolds with nef tangent bundles 
are classified by F. Campana and T. Peternell. 

\begin{thm}[\cite{cp91}]
\label{cp intro}
Let $X$ be a smooth projective threefold 
with nef tangent bundle. 
Then up to taking finite {\'e}tale coverings, 
$X$ is one of the following: 
\begin{enumerate}
\item $\mathbb{P}^3$. 
\item a three dimensional smooth quadric. 
\item $\mathbb{P}^1 \times \mathbb{P}^2$. 
\item $\mathbb{P}^1 \times \mathbb{P}^1 \times \mathbb{P}^1$. 
\item $\mathbb{P}({\mathcal{T}_{\mathbb{P}^2}})$. 
\item $\mathbb{P}_{A}(\mathcal{E})$, 
where $A$ is an Abelian surface and 
$\mathcal{E}$ is a rank two vector bundle 
obtained as an extension of two line bundles 
in $\Pic^0(A)$. 
\item $\mathbb{P}_{C}(\mathcal{E})$, 
where $C$ is an elliptic curve and 
$\mathcal{E}$ is a rank three vector bundle 
obtained as extensions of three line bundles 
of degree zero. 
\item $\mathbb{P}_{C}(\mathcal{E}_{1}) \times_{C} \mathbb{P}_{C}(\mathcal{E}_{2})$, 
where $C$ is an elliptic curve and 
$\mathcal{E}_{i}$ are rank two vector bundles 
obtained as extensions of degree zero line bundles. 
\item an Abelian threefold. 
\end{enumerate}
\end{thm}

Among the above threefolds, 
the existence of Bridgeland stability conditions is known 
in the following cases: 

\begin{itemize}
\item $\mathbb{P}^3$ by \cite{bmt14a, mac14}. 
\item a three dimensional smooth quadric by \cite{sch14}. 
\item (3) -- (5) in Theorem \ref{cp intro} by \cite{bmsz17}. 
\item an Abelian threefold by \cite{bms16, mp16a, mp16b}. 
\end{itemize}

In this paper, we treat the remaining cases, i.e. 
(6) -- (8) in Theorem \ref{cp intro}. 
Note that $\mathbb{P}^1 \times A$, 
$\mathbb{P}^2 \times C$, and  
$\mathbb{P}^1 \times \mathbb{P}^1 \times C$ 
are treated in the author's previous paper \cite{kos17}, 
which are the special cases of (6) -- (8) in Theorem \ref{cp intro}.  

Furthermore, on $\mathbb{P}(\mathcal{T}_{\mathbb{P}^2})$, 
we will construct new Bridgeland stability conditions 
which were not obtained in \cite{bmsz17}.


\subsection{Outline of the proof}
As mentioned in the last subsection, 
we treat the cases (6) -- (8) in Theorem \ref{cp intro} 
in the first part of this paper. 
Recall that, if the bundle is a trivial bundle, 
then the BG type inequality conjecture 
is known to be true by the author's previous paper \cite{kos17}. 
In the trivial bundle case, 
the existence of good endomorphisms is crucial 
in the proof. 

When the bundle is non-trivial, 
we don't know 
the existence of the endomorphisms in general. 
In such cases, 
we use the technique developed by Bayer et al (\cite{bay}), 
which we now explain: 
Consider a smooth family 
$\mathcal{X} \to \mathbb{A}^1$ 
of threefolds over 
the affine line $\mathbb{A}^1$.
Assume that for every points 
$t, t' \in \mathbb{A}^1 \setminus \{0\}$, 
we have 
$\mathcal{X}_{t} \cong \mathcal{X}_{t'}=:X$. 
Then according to \cite{bay}, 
the BG type inequality conjecture for $X$ 
is reduced to that of $\mathcal{X}_{0}$. 
In our situation, 
using this technique, 
we can reduce to the cases 
of the projectivizations of split vector bundles 
(see Proposition \ref{degenerate}). 
Then for split cases, we can argue as in \cite{kos17} 
using good finite morphisms. 

In the second part, 
we will treat the case when $X=\mathbb{P}(\mathcal{T}_{\mathbb{P}^2})$. 
In \cite{bmsz17}, they used the fact that 
$\mathbb{P}(\mathcal{T}_{\mathbb{P}^2})$ 
is a Fano variety 
to construct Bridgeland stability conditions. 
On the other hand, in this paper, 
we regard it as a $\mathbb{P}^1$-bundle 
over $\mathbb{P}^2$ and 
use a full exceptional collection on 
the derived category.


\subsection{Open problems}
\begin{enumerate}
\item As we will see in Conjecture \ref{bms}, 
the conjectural BG type inequality depends on 
a class $B+i\omega \in \NS(X)_{\mathbb{C}}$ 
with $\omega$ ample. 
For threefolds in Theorem \ref{cp intro}, 
except for (5), 
we can prove the inequalty for any choice 
of a class $B+i\omega \in \NS(X)_{\mathbb{C}}$ 
with $\omega$ ample. 

On the other hand, for $\mathbb{P}(\mathcal{T}_{\mathbb{P}^2})$, 
we can prove it only when $B$ and $\omega$ are 
proportional so far. 
We can hope the inequality also holds for any choice 
of $B+i\omega \in \NS\left(
\mathbb{P}(\mathcal{T}_{\mathbb{P}^2})
\right)_{\mathbb{C}}$. 
At this moment, the author doesn't know 
how to solve this problem. 

\item It is expected that 
the space of Bridgeland stability conditions 
has complex dimension equal to 
the rank of the algebraic cohomology 
(In fact, it is true in the surface case 
by the works \cite{ab13, bri08, tod13b}). 
As proven in the paper \cite{bms16}, 
the BG type inequality in 
Conjecture \ref{bms} implies 
the existence of a four dimensional subset 
in the space of Bridgeland stability conditions. 

In \cite[Theorem 3.21]{opt18}, 
the full dimensional family of 
Bridgeland stability conditions on Abelian threefolds
was constructed. 
Proving the same statement for threefolds treated in this paper 
is an interesting open problem, 
which requires the stronger BG type inequality. 

\end{enumerate}


\subsection{Plan of the paper}
In Section \ref{prelim}, we recall about 
the theory of Bridgeland stability conditions 
and about threefolds 
with nef tangent bundles. 
In Section \ref{deform and endo} 
we treat varieties in Theorem \ref{cp intro} (6) -- (8). 
In Section \ref{remaining}, 
we will discuss about the stability conditions 
on $\mathbb{P}(\mathcal{T}_{\mathbb{P}^2})$. 


\begin{ACK}
I would like to thank Professor Arend Bayer and 
my supervisor Professor Yukinobu Toda 
for valuable discussions 
and for careful reading of 
the previous version of this article. 
I would also like to thank 
Yohsuke Matsuzawa and Akihiro Kanemitsu 
who taught me about the papers \cite{cp91, nak02}. 

This paper was written 
while I was visiting 
the University of Edinburgh 
from March 2018 to August 2018. 
This visiting was supported 
by the JSPS program 
``Overseas Challenge Program for Young Researchers''. 
This work was supported by the program 
for Leading Graduate Schools, MEXT, Japan, 
and by Grant-in-Aid for JSPS Research Fellow 17J00664. 

Finally, I would like to thank the referee 
for careful reading of this paper 
and giving me various suggestions. 
\end{ACK}

\begin{NaC}
In this paper we always work over $\mathbb{C}$. 
We use the following notations: 
\begin{itemize}
\item $\ch^{B}=(\ch_{0}^{B}, \cdots, \ch_{n}^{B}):=e^{-B}.\ch$, 
where $\ch$ denotes the Chern character 
and $B \in \NS(X)_{\mathbb{R}}$. 
\item $v^B:=\omega.\ch^{B}
:=(\omega^n.\ch^{B}_{0}, \cdots, \omega.\ch^{B}_{n-1}, \ch^{B}_{n})$, 
where $B, \omega \in \NS(X)_{\mathbb{R}}$.  
\item $K(\mathcal{A})$ : 
the Grothendieck group of an abelian category $\mathcal{A}$. 
\item $\hom(E, F):=\dim\Hom(E, F)$. 
\item $\ext^i(E, F):=\dim\Ext^i(E, F)$. 
\item $D^b(X):=D^b(\Coh(X))$ : 
the bounded derived category of coherent sheaves 
on a smooth projective variety $X$. 
\end{itemize}
\end{NaC}


\section{Preliminaries}
\label{prelim}

\subsection{Bridgeland stability condition}
In this subsection, we recall the notion of 
Bridgeland stability conditions 
on a triangulated category. 
The reference for this subsection 
is Bridgeland's original paper \cite{bri07}. 
First, we define the notion of stability functions: 

\begin{defin}
Let $\mathcal{A}$ be an Abelian category. 
\begin{enumerate}
\item A {\it stability function} on $\mathcal{A}$ 
is a group homomorphism 
$Z \colon K(\mathcal{A}) \to \mathbb{C}$ 
satisfying the condition 
\[
Z (\mathcal{A} \setminus \{0\}) 
\subset \mathcal{H} \cup \mathbb{R}_{<0}, 
\]
where $\mathcal{H}$ is the upper half plane.

\item Let $Z$ be a stability function on $\mathcal{A}$. 
An object $E \in \mathcal{A}$ is called 
$Z$-{\it stable (resp. semistable)} if 
for every non zero proper subobject 
$0 \neq F \subset E$, we have an inequality 
\[
-\frac{\Re Z(F)}{\Im Z(F)} 
< (\mbox{resp. } \leq ) 
-\frac{\Re Z(E)}{\Im Z(E)}. 
\]
Here, we define 
$-\frac{\Re Z(E)}{\Im Z(E)}:=+ \infty$ 
if $\Im Z(E)=0$. 

\item A stability function $Z$ 
on $\mathcal{A}$ satisfies the 
{\it Harder-Narasimhan (HN) property} 
if the following holds: 
for every object 
$E \in\mathcal{A}$, 
there exists a filtration 
\[
0=E_{0} \subset E_{1} \subset \cdots \subset E_{m-1} \subset E_{m}=E 
\]
such that 
$F_{i}:=E_{i}/E_{i-1}$ are $Z$-semistable and 
\[
-\frac{\Re Z(F_{1})}{\Im Z(F_{1})} 
> \cdots > 
-\frac{\Re Z(F_{m})}{\Im Z(F_{m})}. 
\]
\end{enumerate}
\end{defin}

We now define the notion of 
stability conditions on a triangulated category: 

\begin{defin}
Let $\mathcal{D}$ be a triangulated category. 
A {\it stability condition} on $\mathcal{D}$ 
is a pair consisting of the heart $\mathcal{A}$ 
of a bounded t-structure on $\mathcal{D}$ and 
a stability function $Z$ on $\mathcal{A}$ 
satisfying the HN property. 
A stability function $Z$ is called 
a {\it central charge}. 
\end{defin}


\subsection{Bogomolov-Gieseker type inequality conjecture}
In this subsection, we recall the conjectural approach 
for the construction of stability conditions on threefolds. 
Let $X$ be a smooth projective threefold. 
Fix a class 
$B+i\omega \in \NS(X)_{\mathbb{C}}$ 
with $\omega$ ample. 
Conjecturally, there exists a stability condition 
on $D^b(X)$ with its central charge given as follows 
(cf. \cite[Conjecture 2.1.2]{bmt14a}): 
\[
Z_{\omega, B}:=
- \int_{X} e^{-i\omega}.\ch^{B}. 
\]

It is easy to see that the pair 
$(Z_{\omega, B}, \Coh(X))$ 
does not define a stability condition 
when $X$ is a threefold. 
Hence we need to introduce new hearts.
Our hearts are obtained by 
the double-tilting construction \cite{bmt14a} 
which we explain below, 
see the paper \cite{hrs96} for the general theory 
of torsion pairs and tilting. 
In the following, we assume that 
$B \in \NS(X)_{\mathbb{Q}}$ 
and 
$\omega=mH$ 
for some ample divisor $H$ 
and $m \in \mathbb{R}_{>0}$ 
with $m^2 \in \mathbb{Q}$. 
As in the introduction, 
we use the following notation: 
\[
v^B=(v^B_{0}, v^B_{1}, v^B_{2}, v^B_{3})
:=(\omega^3.\ch^{B}_{0}, \omega^2.\ch^B_{1}, \omega.\ch^{B}_{2}, \ch^{B}_{3}). 
\]

{\bf First tilting: }
We define the slope function on 
$\Coh(X)$ as follows: 
\[
\mu_{\omega, B}:= \frac{v^B_{1}}{v^B_{0}}
\colon \Coh(X) \to (-\infty, +\infty]. 
\]

Then define the full subcategories 
$\mathcal{T}_{\omega, B}, 
\mathcal{F}_{\omega, B} \subset \Coh(X)$ 
as follows: 
\begin{align*}
&\mathcal{T}_{\omega, B}
:= \left\langle T \in \Coh(X): 
T \mbox{ is }\mu_{\omega, B} 
\mbox{-semistable with } 
\mu_{\omega, B}(T)>0 
\right\rangle, \\ 
&\mathcal{F}_{\omega, B}
:=\left\langle F \in \Coh(X): 
F \mbox{ is } \mu_{\omega, B} 
\mbox{-semistable  with }
\mu_{\omega, B}(F) \leq 0 
\right\rangle. 
\end{align*}
Here, $\mu_{\omega, B}$-stability for 
coherent sheaves is defined in a standard manner, 
and we denote by $\langle S \rangle$ 
the extension closure of a set of objects 
$S \subset \Coh(X)$. 
Now we define a new heart, 
called tilted heart by 
\[
\Coh^{\omega, B}(X):=\left\langle
\mathcal{F}_{\omega, B}[1], \mathcal{T}_{\omega, B}
\right\rangle. 
\]

{\bf Second tilting: }
As in the first tilting, 
we introduce a new slope function and tilting of 
$\Coh^{\omega, B}(X)$: 
A slope function $\nu_{\omega, B}$ 
on $\Coh^{\omega, B}(X)$ is defined to be 
\[
\nu_{\omega, B}:= 
\frac{v^B_{2}-\frac{1}{6}v^B_{0}}
       {v^B_{1}}
\colon \Coh^{\omega, B}(X) \to (-\infty, +\infty], 
\]
and the notion of $\nu_{\omega, B}$-stability 
for objects in $\Coh^{\omega, B}(X)$ is defined 
similarly as $\mu_{\omega, B}$-stability for coherent sheaves. 
We also refer to $\nu_{\omega, B}$-stability as 
{\it tilt stability}. 
Note that the existence of Harder-Narasimhan filtration 
with respect to $\nu_{\omega, B}$-stability 
is shown in \cite[Lemma 3.2.4]{bmt14a}. 
We define full subcategories of $\Coh^{\omega, B}(X)$ as 
\begin{align*}
&\mathcal{T}^{'}_{\omega, B}
:= \left\langle T \in \Coh^{\omega, B}(X): 
T \mbox{ is }\nu_{\omega, B} 
\mbox{-semistable with } 
\nu_{\omega, B}(T)>0 
\right\rangle, \\ 
&\mathcal{F}^{'}_{\omega, B}
:=\left\langle F \in \Coh^{\omega, B}(X): 
F \mbox{ is } \nu_{\omega, B} 
\mbox{-semistable  with }
\nu_{\omega, B}(F) \leq 0 
\right\rangle. 
\end{align*}

Now we reach the definition of 
the double-tilted heart: 
\[
\mathcal{A}_{\omega, B}:=
\left\langle
\mathcal{F}^{'}_{\omega, B}[1], 
\mathcal{T}^{'}_{\omega, B} 
\right\rangle. 
\]

In the paper \cite{bmt14a}, 
Bayer, Macr{\`i}, and Toda conjectured the following: 

\begin{conj}[{\cite[Conjecture 3.2.6]{bmt14a}}]
\label{bmt}
The pair 
$\left(
Z_{\omega, B}, 
\mathcal{A}_{\omega, B} 
\right)$
is a stability condition on $D^b(X)$. 
\end{conj}

Let us denote 
\[
\overline{\Delta}_{\omega, B}(E)
:=v^B_{1}(E)^2-2v^B_{0}(E)v^B_{2}(E) 
\]
and 
\[
\overline{\nabla}_{\omega, B}(E)
:=2(v_{2}^B(E))^2-3v_{1}^B(E)v_{3}^B(E). 
\]

The following is the so-called 
Bogomolov-Gieseker (BG) type inequality conjecture 
(\cite{bmt14a, bms16, pt15}). 

\begin{conj}[{\cite[Conjecture 3.8]{pt15}}]
\label{original}
For any $\nu_{\omega, B}$-stable object $E$, 
we have the inequality 
\[
\overline{\Delta}_{\omega, B}(E)
+6\overline{\nabla}_{\omega, B}(E)
\geq 0. 
\]
\end{conj}

The BG type inequality conjecture implies 
the existence of a stability condition: 

\begin{prop}[{\cite[Corollary 5.2.4]{bmt14a}}]
Assume that Conjecture \ref{original} holds. 
Then Conjecture \ref{bmt} also holds. 
\end{prop}

\subsection{Reduction theorems}
In this subsection, 
we recall two reduction theorems of the BG type inequality conjecture 
due to \cite{bms16, li18, pt15}. 

First we recall the following notion. 

\begin{defin}
Fix real numbers 
$\alpha_{0} >0$ and $\beta_{0}$. 
Let $E \in \Coh^{\alpha_{0}\omega, B+\beta_{0}\omega}(X)$ be 
a $\nu_{\alpha_{0}\omega, B+\beta_{0}\omega}$-semistable object. 

\begin{enumerate}
\item We define a real number 
$\bar{\beta}(E)$ as 
\[
\bar{\beta}(E):=
\frac{2v^B_{2}(E)}
     {v^B_{1}(E)+\sqrt{\overline{\Delta}_{\omega, B}(E)}}. 
\]

\item $E$ is $\bar{\beta}$-{\it semistable} (resp. {\it stable}) if 
there exists an open neighborhood $V$ of $(0, \bar{\beta}(E))$ 
in the $(\alpha, \beta)$-plane such that for every 
$(\alpha, \beta) \in V$ with $\alpha >0$, 
$E$ is $\nu_{\alpha\omega, B+\beta\omega}$-semistable 
(resp. stable). 
\end{enumerate} 
\end{defin}

The first reduction is of the following form. 

\begin{conj}[{\cite[Conjecture 3.17]{pt15}}]
\label{bms}
Let $E$ be a $\bar{\beta}$-stable object. 
Then we have 
\[
\ch^{B+\bar{\beta}(E)\omega}_{3}(E) \leq 0. 
\]
\end{conj}

\begin{thm}[{\cite[Theorem 3.20]{pt15}}]
Conjectures \ref{original} and \ref{bms} 
are equivalent. 
\end{thm}

Using the same technique, 
the following result was proved in \cite{li18}. 

\begin{thm}[{\cite[Theorem 3.2]{li18}}]
\label{li red}
Let $H$ be an ample divisor on X. 
Assume that there exists a real number 
$\alpha_{0}>0$ such that 
for every real number 
$0 < \alpha <\alpha_{0}$, 
Conjecture \ref{original} is true 
for $(X, \alpha H, B=0)$. 
Then it also holds for $(X, \alpha H, \beta H)$ 
with any choice of $\alpha \geq \frac{1}{2\sqrt{3}}$ 
and $\beta \in \mathbb{R}$. 
\end{thm}


\subsection{Counter-examples}

Counter-examples to Conjecture \ref{bmt} 
are constructed in the papers \cite{kos17, ms19, sch17}. 
In particular, we have the following result: 
 
\begin{lem}[{\cite[Lemma 3.1]{ms19}}]
\label{neg lem}
Let $H$ be an ample divisor. 
Assume that there exists an effective divisor $D$ 
such that 
\begin{equation}
\label{neg}
D^3 
>\frac{(H^2.D)^3}{4(H^3)^2} 
+\frac{3}{4}\frac{(H.D^2)^2}{H^2.D}. 
\end{equation}
Then there exists a pair $(\alpha, \beta)$ 
of real numbers such that 
the pair 
$(Z_{\alpha H, \beta H}, 
\mathcal{A}_{\alpha H, \beta H})$ 
does not define a stability condition. 
\end{lem}

\begin{rmk}
\label{rmk on neg}
Let $D$ be a nef divisor. 
We claim that $D$ does not satisfy the inequality (\ref{neg}). 
By the Hodge index theorem for nef divisors, 
we have the following inequalities: 
\begin{align}
&(H^2.D)^3 \geq (H^3)^2 \cdot D^3 \label{h1} \\
&(H.D^2)^3 \geq H^3 \cdot (D^3)^2. \label{h2}
\end{align}

On the other hand, by replacing $H$ 
with its sufficiently large multiple 
and taking a smooth member, 
the Hodge index theorem on $H$ leads the inequality 
\begin{equation}
\label{h3}
(H^2.D)^2
=(H|_{H}.D|_{H})^2 
\geq (H|_{H})^2 \cdot (D|_{H})^2
=H^3 \cdot H.D^2. 
\end{equation}

The inequality (\ref{h1}) is equivalent to 
the inequality 
\begin{equation}
D^3 \leq \frac{(H^2.D)^3}{(H^3)^2}. 
\label{h4}
\end{equation}

Furthermore, by the inequalities 
(\ref{h2}), (\ref{h3}), and (\ref{h4}), 
we have 
\begin{equation}
\label{h5}
\begin{aligned}
\frac{(H.D^2)^2}{H^2.D} 
&\geq 
\frac{H^3 \cdot (D^3)^2}{H^2.D \cdot H.D^2} 
\quad (\mbox{by }(\ref{h2}))\\
&\geq \frac{(H^2.D)^2}{H.D^2 \cdot H^3}D^3 
\quad (\mbox{by } (\ref{h4})) \\ 
&\geq D^3 
\quad (\mbox{by } (\ref{h3})). 
\end{aligned}
\end{equation}

By combining the inequalities 
(\ref{h4}) and (\ref{h5}), 
we conclude that $D$ satisfies 
the opposite inequality to 
that in (\ref{neg}). 
Hence we can think the inequality (\ref{neg}) 
as a kind of negativity conditions on an effective divisor. 
We can still expect that Conjecture \ref{bmt} 
and Conjecture \ref{original} are true 
if all effective divisors satisfy 
some positivity conditions. 
\end{rmk}


\subsection{Threefolds with nef tangent bundles}
In this subsection, we recall results 
on threefolds with nef tangent bundles, 
which we will need in this paper. 

\begin{prop}[{\cite[Proposition 2.12]{cp91}}]
\label{eff nef}
Let $X$ be a smooth projective variety 
with nef tangent bundle. 
Then every effective divisor on $X$ is nef. 
\end{prop}

The above proposition, 
together with Remark \ref{rmk on neg}, 
shows that there does not exist 
an effective divisor 
on a threefold with nef tangent bundle 
satisfying the inequality (\ref{neg}) 
in Lemma \ref{neg lem}. 
Furthermore, the above proposition 
also ensures the tilt-stability 
of line bundles: 

\begin{lem}[{\cite[Corollary 3.11]{bms16}}]
\label{tilt stab of lb}
Let $X$ be a smooth projective threefold, 
$\omega$ an ample $\mathbb{R}$-divisor on $X$. 
Assume that for every effective divisor $D$ on $X$, 
we have $\omega.D^2 \geq 0$. 
Then for every line bundle $L$ on $X$ and 
$B \in \NS(X)_{\mathbb{R}}$, 
$L$ or $L[1]$ is 
$\nu_{\omega, B}$-stable. 
\end{lem}

Next we recall 
the classification theorem of 
threefolds with nef tangent bundles 
due to the paper \cite{cp91}. 

\begin{thm}[{\cite[Theorem 10.1]{cp91}}]
\label{cp}
Let $X$ be a smooth projective threefold 
with nef tangent bundle. 
Then there exists an {\' e}tale covering 
$f \colon \widetilde{X} \to X$ such that 
$\widetilde{X}$ is one of the following: 
\begin{enumerate}
\item $\mathbb{P}^3$. 
\item a three dimensional smooth quadric. 
\item $\mathbb{P}^1 \times \mathbb{P}^2$. 
\item $\mathbb{P}^1 \times \mathbb{P}^1 \times \mathbb{P}^1$. 
\item $\mathbb{P}({\mathcal{T}_{\mathbb{P}^2}})$. 
\item $\mathbb{P}_{A}(\mathcal{E})$, 
where $A$ is an Abelian surface and 
$\mathcal{E}$ is a rank two vector bundle 
obtained as an extension of two line bundles 
in $\Pic^0(A)$. 
\item $\mathbb{P}_{C}(\mathcal{E})$, 
where $C$ is an elliptic curve and 
$\mathcal{E}$ is a rank three vector bundle 
obtained as extensions of three line bundles 
of degree zero. 
\item $\mathbb{P}_{C}(\mathcal{E}_{1}) \times_{C} \mathbb{P}_{C}(\mathcal{E}_{2})$, 
where $C$ is an elliptic curve and 
$\mathcal{E}_{i}$ are rank two vector bundles 
obtained as extensions of degree zero line bundles. 
\item an Abelian threefold. 
\end{enumerate}
\end{thm}

For our purpose, we need the following observation: 

\begin{lem}
\label{gal}
In Theorem \ref{cp}, we can choose 
an {\' e}tale covering $f$ to be a Galois covering. 
\end{lem}

\begin{proof}
Let $X$ be a smooth projective threefold 
with nef tangent bundle. 
In the proof of \cite[Theorem 10.1]{cp91}, 
they actually show the existence of 
the following diagram of smooth projective varieties: 
\[
\xymatrix{
&\widetilde{X}:=\widetilde{Y} \times_{Y} X \ar[rr]^{f} \ar[d] & &X \ar[d] \\
&\widetilde{Y} \ar[r]_{\psi} &Y' \ar[r]_{\phi} &Y, 
}
\] 
where $Y'$ is an Abelian variety 
(possibly of dimension zero), 
$\psi$ and $\phi$ are {\' e}tale coverings. 
Note that the morphism $\widetilde{X} \to \widetilde{Y}$ 
is as (1) -- (9) in Theorem \ref{cp}, i.e., 
$\widetilde{Y}$ is $\Spec\mathbb{C}$, $A$, $C$, or an Abelian threefold 
in the notation of Theorem \ref{cp}. 

Put $g:=\phi \circ \psi$. 
Let us take the Galois closure of $g$, i.e. 
an {\' e}tale covering 
$h \colon \widehat{Y} \to \widetilde{Y}$ 
such that the morphism 
$h \circ g \colon \widehat{Y} \to Y$ 
is an {\' e}tale Galois covering. 
Note that since $\widetilde{Y}$ 
is an Abelian variety, so is $\widehat{Y}$. 
Hence the base change 
$\widehat{X}:=\widehat{Y} \times_{Y} X$ 
is again one of the threefolds in Theorem \ref{cp} (1) -- (9), 
and is an {\' e}tale Galois covering of $X$. 
This completes the proof. 
\end{proof}

\begin{rmk}
Among threefolds in Theorem \ref{cp}, 
Conjecture \ref{bms} is known to be true 
in the following cases: 
\begin{itemize}
\item $\mathbb{P}^3$ by \cite{bmt14a, mac14}. 
\item a three dimensional smooth quadric by \cite{sch14}. 
\item $\mathbb{P}^1 \times \mathbb{P}^2, 
\mathbb{P}^1 \times \mathbb{P}^1 \times \mathbb{P}^1$ 
with any choice of a class $B+i\omega$ 
by \cite{bmsz17} 
(In the paper \cite{bmsz17}, they only treat 
the case when $B$ and $\omega$ are proportional. 
Even when they are not proportional, 
the same proof works according to the formulation 
given in Conjecture \ref{bms}). 
\item $\mathbb{P}({\mathcal{T}_{\mathbb{P}^2}})$ 
with $B$ and $\omega$ being proportional to 
the anti-canonical class 
by \cite{bmsz17}. 
\item an Abelian threefold 
with any choice of a class $B+i\omega$ 
by \cite{bms16, mp16a, mp16b}. 
\end{itemize}
\end{rmk}

The following is our first main result, 
which completely solves Conjecture \ref{bms} 
for threefolds as in Theorem \ref{cp} (6) -- (8): 

\begin{thm}
\label{first thm}
Let $X$ be a threefold as in 
Theorem \ref{cp} (6), (7), or (8). 
Then for every class 
$B+i\omega \in \NS(X)_{\mathbb{C}}$ 
with $\omega$ ample, 
Conjecture \ref{bms} holds. 
\end{thm}

As a corollary, we obtain: 

\begin{cor} 
\label{first cor}
Let $X$ be a smooth projective threefold with 
nef tangent bundle. 
Then there exist Bridgeland stability conditions 
on $D^b(X)$. 
\end{cor}

\begin{proof}
By \cite[Proposition 6.1]{bms16}, 
we may replace $X$ by an {\' e}tale Galois covering, 
thus we can assume it is one of the threefolds 
in Theorem \ref{cp} 
(see Lemma \ref{gal}). 
Then Theorem \ref{first thm}, 
together with the previous works 
\cite{bms16, bmsz17, mp16a, mp16b, mac14, sch14}, 
proves the required statement. 
\end{proof} 

We will also have the following result for 
$X=\mathbb{P}(\mathcal{T}_{\mathbb{P}^2})$: 

\begin{thm}
\label{second thm}
Let $X=\mathbb{P}(\mathcal{T}_{\mathbb{P}^2})$, 
$H$ be an ample divisor on $X$. 
Let $\alpha >\frac{1}{2\sqrt{3}}$ and $\beta \in \mathbb{R}$ be 
real numbers. 
Then Conjecture \ref{original} holds for 
$(X, \alpha H, \beta H)$. 

\end{thm}


\section{Proof of Theorem \ref{first thm}}
\label{deform and endo}
In this section, 
we prove Theorem \ref{first thm}. 
We use the following terminology. 

\begin{defin}
Let $X$ be as in Theorem \ref{cp} (6) -- (8). 
Then $X$ is {\it split} if 
the vector bundles defining $X$ 
are direct sums of line bundles. 
\end{defin}


\subsection{Reduction to split cases}
In this subsection, we reduce 
Theorem \ref{first thm} to the split cases. 
The key method is the following result 
due to the paper \cite{bay}. 

\begin{prop}[{\cite[Proposition 27.1]{bay}}]
\label{family}
Let $f \colon \mathcal{X} \to D$ 
be a smooth projective family of threefolds 
over a smooth curve $D$ 
and fix a point $0 \in D$. 
Suppose that $f$ is a trivial family 
over $U:=D \setminus \{0\}$, 
i.e. $f^{-1}(U) \cong X \times U$ 
for some threefold $X$. 
Take an $f$-ample $\mathbb{Q}$-divisor $\mathcal{H}$ 
and an arbitrary $\mathbb{Q}$-divisor $\mathcal{B}$ on 
$\mathcal{X}$. 
Let $\mathcal{H}_{0}$, $\mathcal{B}_{0}$ 
(resp. H, B) 
be restriction of $\mathcal{H}$, $\mathcal{B}$ 
to the special fiber $f^{-1}(0)$ 
(resp. the general fiber $X$). 
If Conjecture \ref{bms} is true for 
$(f^{-1}(0), \mathcal{H}_{0}, \mathcal{B}_{0})$, 
then it also holds for 
$(X, H, B)$. 
\end{prop}

The above result follows from the existence of 
the relative moduli spaces of tilt-stable objects 
over the base $D$, 
satisfying the valuative criterion for 
universal closedness. 

\begin{prop}
\label{degenerate}
Assume that Theorem \ref{first thm} holds 
for every split $X$. 
Then it also holds for every non-split $X$. 
\end{prop} 

\begin{proof}
First we consider the case (6) in Theorem \ref{cp}. 
Let $A$ be an Abelian surface, 
$\mathcal{E}$ be a rank two vector bundle 
which fits into the non-split short exact sequence 
\[
0 \to \mathcal{O}_{A} \to \mathcal{E} \to L \to 0 
\]
with $L \in \Pic^0(A)$. 
Note that we may assume $L=\mathcal{O}_A$, 
as this is the only case when we have 
$\Ext^1(L, \mathcal{O}_A) \neq 0$. 
Let $X:=\mathbb{P}_{A}(\mathcal{E})$. 
By applying Proposition \ref{family}, 
we will show that 
to prove Theorem \ref{first thm} for $X$, 
it is enough to show it for 
$X_{0}:=\mathbb{P}^1 \times A$. 
Let us take an affine line 
$\mathbb{A}^1 \subset 
\Ext^1(\mathcal{O}_{A}, \mathcal{O}_{A})$ 
passing through the origin 
and a point 
$[\mathcal{E}] \in \Ext^1(\mathcal{O}_{A}, \mathcal{O}_{A})$. 
Over $\mathbb{A}^1$, we have a smooth family 
$f \colon \mathcal{X} \to \mathbb{A}^1$ 
with the following properties 
(cf. \cite[Lemma 4.1.2]{hl97}): 
\begin{enumerate}
\item Let $U:=\mathbb{A}^1 \setminus \{0\}$. 
Then we have 
$\mathcal{X}_{U}:=f^{-1}(U) 
\cong X \times U$. 
\item We have 
$\mathcal{X}_{0}:=f^{-1}(0) \cong X_{0}$. 
\end{enumerate}

Indeed, the family 
is constructed as a $\mathbb{P}^1$-bundle 
$\sigma \colon 
\mathcal{X}=\mathbb{P}_{A \times \mathbb{A}^1}(\mathcal{U}) 
\to A \times \mathbb{A}^1$, 
where $\mathcal{U}$ fits into the exact sequence 
\[
0 \to \mathcal{U} \to q^*\mathcal{E} \to i_{*}\mathcal{O}_A \to 0. 
\]
Here, $q \colon A \times \mathbb{A}^1 \to A$ is a projection 
and $i \colon A \times \{0\} \to A \times \mathbb{A}^1$ 
is an inclusion. 
Let $p:=q \circ \sigma \colon \mathcal{X} \to A$ be a projection. 
We also have to prove that, 
for a given ample divisor $H$ on $X$, 
there exists an $f$-ample divisor 
$\mathcal{H}$ on $\mathcal{X}$ 
such that its restriction to $X$ 
coincides with $H$. 
By Lemma \ref{ns}, 
we can write as 
$H=\mathcal{O}_{\pi}(a) \otimes \pi^*N$, 
where $\pi \colon X \to A$ is a 
structure morphism, 
$N$ is an ample line bundle on $A$, 
and $a >0$ is a positive integer. 
We put $\mathcal{H}:=\mathcal{O}_{\sigma}(a) \otimes p^*N$. 
Then the restriction 
$\mathcal{H}|_{f^{-1}(0)} \cong 
\mathcal{O}_{\mathbb{P}^1}(a) \boxtimes N$ 
to the central fiber 
$\mathbb{P}^1 \times A$ is ample. 
Hence the line bundle $\mathcal{H}$ 
is ample on each fiber of $f$, 
i.e., it is $f$-ample. 
Now the result holds by Proposition \ref{family}. 

Next let $C$ be an elliptic curve and 
$L_{i}$ be degree zero line bundles on $C$ 
($i=1, 2, 3$). 
Consider the case (7) in Theorem \ref{cp}, 
i.e. $X=\mathbb{P}_{C}(\mathcal{E})$, 
where $\mathcal{E}$ is a rank three vector bundle 
obtained as follows: 
\begin{align*}
& 0 \to L_{1} \to \mathcal{E}' \to L_{2} \to 0, \\
& 0 \to \mathcal{E}' \to \mathcal{E} \to L_{3} \to 0. 
\end{align*}

As above, by considering a family over the affine line 
$\mathbb{A}^1 \subset \Ext^1(L_{3}, \mathcal{E}')$ 
passing through the origin and a class $[\mathcal{E}]$, 
we may assume that 
$\mathcal{E}=\mathcal{E}' \oplus L_{3}$. 
Then by applying the same argument for 
$[\mathcal{E}'] \in \Ext^1(L_{2}, L_{1})$, 
we can reduce to the split case. 

Finally, consider the case (8) in Theorem \ref{cp}. 
For $i=1, 2$, let 
$\pi_{i} \colon Y_{i}:=\mathbb{P}_{C}(\mathcal{E}_{i}) \to C$, 
where $\mathcal{E}_{i}$ are rank two vector bundles 
fitting into the short exact sequences 
\[
0 \to \mathcal{O}_{C} \to \mathcal{E}_{i} \to L_{i} \to 0.  
\]
Let $X:=Y_{1} \times_{C} Y_{2}$. 
Noting that $X=\mathbb{P}_{Y_{1}}(\pi_{1}^*\mathcal{E}_{2})$, 
we can first reduce to the case when 
$\mathcal{E}_{2}=\mathcal{O}_{C} \oplus L_{2}$. 
Then by regarding as $X=\mathbb{P}_{Y_{2}}(\pi_{2}^*\mathcal{E}_{1})$, 
we can reduce to the case when $X$ is split. 
\end{proof}


\subsection{Conclusion}
In this subsection, 
we explain how to prove Theorem \ref{first thm} 
in the split cases. 
We use the following notations: 
\begin{itemize}
\item $A$ is an Abelian surface, 
$C$ is an elliptic curve. 
\item $L \in \Pic^0(A)$ 
and 
$L_{1}, L_{2} \in \Pic^0(C)$. 
\item For $m \in \mathbb{Z}_{>0}$, 
$L^{\frac{1}{m}}$ is a line bundle 
such that $(L^{\frac{1}{m}})^m \cong L$. 
$L^{\frac{1}{m}}_{i} \in \Pic^0(C)$ 
are similarly defined. 
\item For $i=1, 2$, 
$Y_{i}:=\mathbb{P}_{C}(\mathcal{O}_{C} \oplus L_{i})$. 
\item $X$ is 
$\mathbb{P}_{A}(\mathcal{O}_{A} \oplus L)$, 
$\mathbb{P}_{C}(\mathcal{O}_{C} \oplus L_{1} \oplus L_{2})$, 
or $Y_{1} \times_{C} Y_{2}$. 
\item For $m \in \mathbb{Z}_{>0}$, 
$Y_{i}^{(m)}:=\mathbb{P}_{C}(\mathcal{O}_{C} \oplus L_{i}^{m})$, 
and 
$Y_{i}^{(\frac{1}{m})}:=
\mathbb{P}_{C}(\mathcal{O}_{C} \oplus L_{i}^{\frac{1}{m}})$. 
$X^{(m)}$, $X^{(\frac{1}{m})}$ are defined similarly. 
\end{itemize}

We start with the following easy lemma: 

\begin{lem}
\label{iden cohom}
Let $X$ be as in Theorem \ref{cp} (6) -- (8) 
which is split, 
let $m \in \mathbb{Z}_{>0}$ be an positive integer. 
Then by identifying the tautological classes, 
we have a ring isomorphism 
\[
\Phi \colon 
H^{2*}(X^{(\frac{1}{m})}, \mathbb{Q}) \to 
H^{2*}(X, \mathbb{Q})
\]
between the even cohomology rings. 
\end{lem}
\begin{proof}
We only treat the case when 
$X=\mathbb{P}_{A}(\mathcal{O}_{A} \oplus L)$. 
Let $h \in H^2(X, \mathbb{Q})$ 
(resp. $h^{(\frac{1}{m})} \in 
H^2(X^{(\frac{1}{m})}, \mathbb{Q})$) 
be a divisor such that 
$\mathcal{O}_{X}(h)=\mathcal{O}_{\pi}(1)$ 
(resp. $\mathcal{O}_{X^{(\frac{1}{m})}}(h^{(\frac{1}{m})})
=\mathcal{O}_{\pi^{(\frac{1}{m})}}(1)$). 
Since $L \in \Pic^0(A)$, 
we have ring isomorphisms 
\[
\Phi \colon 
H^{2*}(X^{(\frac{1}{m})}, \mathbb{Q}) \cong 
H^{2*}(A, \mathbb{Q})[t]/(t^2) \cong 
H^{2*}(X, \mathbb{Q}). 
\]
Here, the isomorphism 
$H^{2*}(A, \mathbb{Q})[t]/(t^2) \cong 
H^{2*}(X, \mathbb{Q})$ 
sends $t$ to $[h]$ 
and the same is true for $X^{(\frac{1}{m})}$. 
Hence $\Phi([h^{(\frac{1}{m})}])=[h]$. 
\end{proof}

Next we construct finite morphisms 
which play important roles for our purpose.

\begin{prop}[cf. {\cite[Proposition 5]{nak02}}]
\label{endo}
Let $X$ be a threefold as in 
Theorem \ref{first thm} 
which is split. 
Then, for every positive integer 
$m \in \mathbb{Z}_{>0}$, 
we have the following commutative diagram 
\begin{equation}
\label{frob diagram}
\xymatrix{
X^{(\frac{1}{m})} \ar@/_/[rdd]_{\pi^{(\frac{1}{m}})} \ar@/^/[rrd]^{F_{m}} \ar[rd]^{g_{m}} & & \\
 & X^{(m)} \ar[d]^{\pi^{(m)}} \ar[r]^{h_{m}} & X \ar[d]^{\pi} \\
 & Z \ar[r]_{\underline{m}} & Z, 
}
\end{equation}
where $Z$ is an Abelian surface $A$ or 
an elliptic curve $C$. 

Furthermore, 
the pull-back via the morphism 
$F_{m} \colon X^{(\frac{1}{m})} \to X$ 
acts on the even cohomology as follows: 
\begin{equation} \label{act}
\Phi \circ F_{m}^* \colon 
H^{2*}(X, \mathbb{Q})  
\ni (x, y, z, w) \mapsto 
(x, m^2y, m^4z, m^6w) 
\in H^{2*}(X, \mathbb{Q}). 
\end{equation}
\end{prop}
\begin{proof}
First consider the case (6) in Theorem \ref{cp}: 
$X:=\mathbb{P}_{A}(\mathcal{O}_{A} \oplus L)$. 
Consider the multiplication map 
$\underline{m} \colon A \to A$. 
By \cite[p. 71 $\rm(i\hspace{-.08em}i\hspace{-.08em}i)$]{mum08}, we have 
$\underline{m}^*L \cong L^m$. 
Hence by base change, 
we have the morphism 
$h_{m} \colon X^{(m)} \to X$. 
On the other hand, 
the natural inclusion 
\begin{equation}
\label{incl}
\mathcal{O}_{A} \oplus L^{m} \subset 
\Sym^{m^2}(\mathcal{O}_{A} \oplus L^{\frac{1}{m}}) 
= \mathcal{O}_{A} \oplus L^{\frac{1}{m}} \oplus 
\cdots \oplus (L^{\frac{1}{m}})^{m^2} 
\end{equation}
induces a morphism 
$g_{m} \colon X^{(\frac{1}{m})} \to X^{(m)}$. 
Now we get a commutative diagram 
as in (\ref{frob diagram}). 
Locally over $A$, the morphism $g_{m}$ 
is nothing but the toric Frobenius morphism 
$\underline{m}^2 \colon \mathbb{P}^1 \to \mathbb{P}^1$. 
In particular, we have 
$g_m^*\mathcal{O}_{\pi^{(m)}}(1)
=\mathcal{O}_{\pi^{(\frac{1}{m})}}(m^2)$. 

To see that the pull back $F_{m}^*$ acts 
on the cohomology as (\ref{act}), 
it is enough to look at the action on 
$H^2(X, \mathbb{Q}) \cong 
\pi^*H^2(A, \mathbb{Q}) \oplus \mathbb{Q} \cdot h$, 
where $h$ is the tautological class. 
For a class $y \in H^2(A, \mathbb{Q})$, 
we have 
$\underline{m}^*y=m^2y$ and hence 
\[
\Phi \circ F_m^* (\pi^*y)
=\Phi\left(\pi^{(\frac{1}{m})*}\underline{m}^* y \right)
=m^2\pi^*(y). 
\]
On the other hand, we have 
\[
\Phi \circ F_m^* (h)
=\Phi\left( g_m^*h^{(m)} \right)
=\Phi\left(m^2 h^{(\frac{1}{m})} \right)
=m^2h. 
\]
Here, the second equality follows from 
the local description of the morphism $g_m$, 
while the third equality follows from the definition of $\Phi$.

Next consider the case (7) in Theorem \ref{cp}, i.e., 
$X=\mathbb{P}_{C}(\mathcal{O}_{C} \oplus L_{1} \oplus L_{2})$. 
Replacing (\ref{incl}) by the inclusion 
\[
\mathcal{O}_{C} \oplus L_{1}^{m} \oplus L_{2}^{m} 
\subset 
\Sym^{m^2}(\mathcal{O}_{C} \oplus L_{1}^{\frac{1}{m}} \oplus L_{2}^{\frac{1}{m}}), 
\]
we get the diagram as in (\ref{frob diagram}). 

Finally, consider the case (8) in Theorem \ref{cp}: 
$X=Y_{1} \times_{C} Y_{2}$. 
As above, we can construct the morphisms 
$Y_{i}^{(\frac{1}{m})} \to Y_{i}^{(m)}$, 
which induce the morphism 
$g_{m} \colon X^{(\frac{1}{m})} \to X^{(m)}$. 
Hence we get the diagram as in (\ref{frob diagram}). 
\end{proof}

\begin{rmk}
\label{nonendo}
By using the inclusion 
\[
\mathcal{O}_{A} \oplus L^{m} \subset 
\Sym^{m}(\mathcal{O}_{A} \oplus L) 
\] 
instead of (\ref{incl}), 
we get an endomorphism 
$F_{m}' \colon X \to X$ 
which is the multiplication map $\underline{m} \colon A \to A$ 
on the base, and the toric Frobenius morphism 
$\underline{m} \colon \mathbb{P}^1 \to \mathbb{P}^1$ 
on the fiber. 
It seems natural to use the endomorphism $F_{m}'$ 
rather than $F_{m}$. 
The issue is that the endomorphism 
$F_{m}'$ is not polarized, i.e., 
there does not exist any ample divisor $H$ on $X$ 
such that the pull back 
$F_m'^*H$ is a multiple of $H$. 
On the other hand, 
the morphism $F_{m}$ behaves like a polarized 
endomorphism, although it is not an endomorphism 
(see the formula (\ref{act})). 
\end{rmk}

According to the description (\ref{act}) 
of the pull back $F_{m}^*$, 
we can prove the following two results: 

\begin{prop}[\cite{bms16}]
\label{approx ext}
Let $X$ be as in 
Theorem \ref{cp} (6), (7), or (8) 
which is split, 
let $F_{m}$ be the morphism 
constructed in Proposition \ref{endo}. 
Let $E \in D^b(X)$ be a two term complex 
concentrated in degree $-1$ and $0$. 
\begin{enumerate}
\item If there exists an ample divisor $H$ on $X^{(\frac{1}{m})}$ 
such that 
\[
\hom
\left(
\mathcal{O}(H), F_{m}^*E
\right)
=0, 
\]
then we have 
\[
\hom
\left(
\mathcal{O}, F_{m}^*E
\right)
=O(m^4). 
\]

\item If there exists an ample divisor $H$ on $X^{(\frac{1}{m})}$ 
such that 
\[
\ext^2
\left(
\mathcal{O}(-H), F_{m}^*E
\right)
=0, 
\]
then 
\[
\ext^2
\left(
\mathcal{O}, F_{m}^*E
\right)
=O(m^4). 
\]
\end{enumerate}
\end{prop}

\begin{proof}
Since we know that 
the pull back $F_{m}^*$ 
acts on the cohomology 
as in (\ref{act}), 
the arguments of 
Section 7 in \cite{bms16} 
prove the result. 
\end{proof} 

\begin{lem}
\label{chern comp}
Let $X$ be as in Theorem \ref{cp} (6) -- (8) 
which is split, 
$m, q \in \mathbb{Z}_{>0}$ be positive integers. 
Take a divisor $D$ on $X$ 
and let $D^{(\frac{1}{mq})}$ 
be a divisor on $X^{(\frac{1}{mq})}$ 
such that 
$D^{(\frac{1}{mq})}=\Phi^{-1}(D)$ 
in the cohomology ring. 
Then for every object $E \in D^b(X)$,  
we have the equality 
\[
\ch_{3}\left(F_{mq}^*E \otimes 
\mathcal{O}(-m^2qD^{(\frac{1}{mq})})
\right)
=m^6q^6\ch_{3}^{\frac{1}{q}D}(E) 
\in \mathbb{Q}
\]
as rational numbers. 
\end{lem}
\begin{proof}
Note that $\Phi(D^{(\frac{1}{mq})})=D$ 
by definition. 
Hence by the formula (\ref{act}), 
we have 
\begin{align*}
& \quad \ch_{3}\left(F_{mq}^*E \otimes 
\mathcal{O}(-m^2qD^{(\frac{1}{mq})})
\right) \\
&=\Phi\left(\ch_{3}\left(F_{mq}^*E \otimes 
\mathcal{O}(-m^2qD^{(\frac{1}{mq})})
\right)
\right) \\
&=-\frac{1}{6}m^6q^3D^{3}\ch_{0}(E)
+\frac{1}{2}\left(m^4q^2D^{2}
\right)\left(m^2q^2\ch_{1}(E)
\right)
-\left(m^2qD
\right)\left(m^4q^4\ch_{2}(E)
\right) \\
&\quad +m^6q^6\ch_{3}(E) \\
&=m^6q^6\ch_{3}^{\frac{1}{q}D}(E)
\end{align*}
as required. 
\end{proof}

Next we prove a variant of 
the toric Frobenius splitting of line bundles. 

\begin{prop}[cf. \cite{tho00}]
\label{toric split}
Let $X$ and $g_{m}$ be as in 
Proposition \ref{endo}. 
Let $M$ be a line bundle on $X^{(\frac{1}{m})}$. 
Then the vector bundle 
$g_{m*}M$ decomposes into a direct sum of line bundles. 
Furthermore, the direct summands are 
explicitly described as follows: 
\begin{enumerate}
\item When $X=\mathbb{P}_{A}(\mathcal{O}_{A} \oplus L)$ 
is as in Theorem \ref{cp} (6) 
and 
$M=
\mathcal{O}_{\pi^{(\frac{1}{m})}}(a) 
\otimes \pi^{(\frac{1}{m})*}N$, 
then each direct summand of $g_{m*}M$ is of the following form: 
\[
\mathcal{O}_{\pi^{(m)}}(i) 
\otimes \pi^{(m)*} \left(L^{\frac{j}{m}} \otimes N \right), 
\]
where 
$i = \lfloor \frac{a}{m^2} \rfloor-1, 
\lfloor \frac{a}{m^2} \rfloor$, 
$0 \leq j \leq m^2$. 

\item When 
$X=\mathbb{P}_{C}(\mathcal{O}_{C} \oplus L_{1} \oplus L_{2})$ 
is as in Theorem \ref{cp} (7) 
and 
$M=
\mathcal{O}_{\pi^{(\frac{1}{m})}}(a) 
\otimes \pi^{(\frac{1}{m})*}N$, 
then each direct summand of $g_{m*}M$ is of the following form: 
\[
\mathcal{O}_{\pi^{(m)}}(i) 
\otimes 
\pi^{(m)*} \left( 
L_{1}^{\frac{j_{1}}{m}} 
\otimes L_{2}^{\frac{j_{2}}{m}} 
\otimes N 
\right), 
\]
where 
$i = \lfloor \frac{a}{m^2} \rfloor-2, 
\lfloor \frac{a}{m^2} \rfloor-1, 
\lfloor \frac{a}{m^2} \rfloor$, 
and 
$0 \leq j_{1}, j_{2} \leq m^2$.

\item  When 
$X=\mathbb{P}_{C}(\mathcal{O}_{C} \oplus L_{1}) 
\times_{C} \mathbb{P}_{C}(\mathcal{O}_{C} \oplus L_{2})$ 
is as in Theorem \ref{cp} (8) 
and 
$M=\mathcal{O}_{\pi^{(\frac{1}{m})}}(a, b) 
\otimes \pi^{(\frac{1}{m})*}N$, 
then each direct summand of $g_{m*}M$ is of the following form: 
\[
\mathcal{O}_{\pi^{(m)}}(i, j) 
\otimes 
\pi^{(m)*} \left( 
L_{1}^{\frac{k_{1}}{m}} 
\otimes L_{2}^{\frac{k_{2}}{m}} 
\otimes N 
\right), 
\]
where 
$i = \lfloor \frac{a}{m^2} \rfloor-1, 
\lfloor \frac{a}{m^2} \rfloor$, 
$j = \lfloor \frac{b}{m^2} \rfloor-1, 
\lfloor \frac{b}{m^2} \rfloor$, 
and 
$0 \leq k_{1}, k_{2} \leq m^2$. 
\end{enumerate}
\end{prop}
\begin{proof}
(1) Let $X=\mathbb{P}_{A}(\mathcal{O}_{A} \oplus L)$ 
be as in Theorem \ref{cp} (6). 
Since $g_{m*}M \cong 
g_{m*}\mathcal{O}_{\pi^{(\frac{1}{m})}}(a) 
\otimes \pi^{(m)*}N$, 
we may assume that 
$M=\mathcal{O}_{\pi^{(\frac{1}{m})}}(a)$. 
Furthermore, since we have 
$g_{m}^*\mathcal{O}_{\pi^{(m)}}(1) 
\cong \mathcal{O}_{\pi^{(\frac{1}{m})}}(m^2)$, 
we may assume 
$0 \leq a < m^2$. 
Now let 
$\mathcal{F}:=g_{m*}M$, 
and consider the adjoint map 
$\alpha \colon 
\pi^{(m)*}\pi^{(m)}_{*}\mathcal{F} \to \mathcal{F}$. 
On the fiber of $\pi^{(m)}$, 
the map $\alpha$ is nothing but the natural inclusion 
\[
\mathcal{O}_{\mathbb{P}^1}^{\oplus a+1} 
\subset 
\mathcal{O}_{\mathbb{P}^1}^{\oplus a+1} 
\oplus \mathcal{O}_{\mathbb{P}^1}(-1)^{\oplus (m^2-a-1)}. 
\]

Indeed by \cite{tho00}, on the fiber of $\pi^{(m)}$, 
we have an isomorphism 
\[
\mathcal{F}|_{\mathbb{P}^1} 
\cong \underline{m}^2_{*}\mathcal{O}_{\mathbb{P}^1}(a) 
\cong 
\mathcal{O}_{\mathbb{P}^1}^{\oplus a+1} 
\oplus \mathcal{O}_{\mathbb{P}^1}(-1)^{\oplus (m^2-a-1)}, 
\]
where $\underline{m}^2$ denotes the 
toric Frobenius morphism on $\mathbb{P}^1$ 
(see also \cite[Theorem 5.2]{bmsz17} 
for this formula). 
Moreover, the adjoint map $\alpha$ restricted to the fiber 
is nothing but the evaluation map 
\[
\alpha|_{\mathbb{P}^1} \colon 
H^0(\mathbb{P}^1, \mathcal{F}|_{\mathbb{P}^1}) 
\otimes \mathcal{O}_{\mathbb{P}^1} 
\hookrightarrow 
\mathcal{F}|_{\mathbb{P}^1}. 
\]

Hence globally, the map $\alpha$ is injective 
and we get the short exact sequence 
\begin{equation}
\label{ses}
0 \to \pi^{(m)*}\pi^{(m)}_{*}\mathcal{F} 
\to \mathcal{F} 
\to \pi^{(m)*}\mathcal{G} \otimes \mathcal{O}_{\pi^{(m)}}(-1) 
\to 0 
\end{equation}
for some coherent sheaf $\mathcal{G} \in \Coh(A)$. 
First of all, we have 
\[
\pi^{(m)}_{*}\mathcal{F}
=\pi^{(\frac{1}{m})}_{*}\mathcal{O}_{\pi^{(\frac{1}{m})}}(a) 
=\Sym^a(\mathcal{O}_{A} \oplus L^{\frac{1}{m}}) 
=\mathcal{O}_{A} \oplus L^{\frac{1}{m}} \oplus \cdots \oplus L^{\frac{a}{m}}. 
\]

Next we will show that 
$\mathcal{G}$ is a direct sum of line bundles. 
Applying the functor 
$\pi^{(m)}_{*}(- \otimes \mathcal{O}_{\pi^{(m)}}(1))$ 
to the exact sequence (\ref{ses}), we have 
\begin{equation*}
\xymatrix{
&0 \to \Sym^a(\mathcal{O}_{A} \oplus L^{\frac{1}{m}}) 
\otimes (\mathcal{O}_A \oplus L^{m}) 
\ar[r]^{\beta} &\Sym^{a+m^2}(\mathcal{O}_{A} \oplus L^{\frac{1}{m}}) 
\to \mathcal{G} \to 0. 
}
\end{equation*}

Note that the vector bundles 
$\Sym^a(\mathcal{O}_{A} \oplus L^{\frac{1}{m}}) 
\otimes (\mathcal{O} \oplus L^{m})$ 
and 
$\Sym^{a+m^2}(\mathcal{O}_{A} \oplus L^{\frac{1}{m}})$ 
are the direct sums of line bundles. 
By the definition of the morphism $g_{m}$, 
the map $\beta$ is the natural inclusion 
as the direct summand. 
Hence $\mathcal{G}$ is isomorphic 
to the vector bundle 
\[
L^{\frac{a+1}{m}} \oplus L^{\frac{a+2}{m}} \oplus \cdots 
\oplus L^{\frac{m^2-2}{m}} \oplus L^{\frac{m^2-1}{m}}.
\]

It remains to show that 
the exact sequence (\ref{ses}) splits. 
Let us first compute the $\Ext$-group: 
\begin{align*}
&\quad \Ext^1\left(
\pi^{(m)*}\mathcal{G} \otimes \mathcal{O}_{\pi^{(m)}}(-1), 
\pi^{(m)*} \pi^{(m)}_{*} \mathcal{F} 
\right) \\ 
&\cong 
H^1\left(X^{(m)}, 
\pi^{(m)*}\left(
\mathcal{G}^{\vee} \otimes \pi^{(m)}_{*} \mathcal{F}
\right) 
\otimes \mathcal{O}_{\pi^{(m)}}(1)
\right) \\ 
&\cong 
H^1\left(A, 
\mathcal{G}^{\vee} \otimes \pi^{(m)}_{*} \mathcal{F} 
\otimes \pi^{(m)}_{*}\mathcal{O}_{\pi^{(m)}}(1)
\right) \\
&\cong 
\bigoplus_{\eta} H^1(A, L^{\frac{\eta}{m}}). 
\end{align*}

Here, the last isomorphism follows from 
the descriptions of 
$\mathcal{G}, \pi^{(m)}_{*}\mathcal{F}$ 
given above, together with the equality 
$\pi^{(m)}_{*}\mathcal{O}_{\pi^{(m)}}(1)
=\mathcal{O}_{A} \oplus L^{m}$. 
Furthermore, by these descriptions, 
we can see that $\eta \neq 0$. 
Hence if $L$ is not a torsion line bundle, 
we have the vanishing 
$\Ext^1\left(
\pi^{(m)*}\mathcal{G} \otimes \mathcal{O}_{\pi^{(m)}}(-1), 
\pi^{(m)*} \pi^{(m)}_{*} \mathcal{F} 
\right)=0$ 
and thus the sequence (\ref{ses}) splits. 
Assume that $L^{\frac{1}{m}}$ 
is $l$-torsion, i.e., 
$(L^{\frac{1}{m}})^l \cong \mathcal{O}_{A}$. 
Assume also that 
$\Ext^1\left(
\pi^{(m)*}\mathcal{G} \otimes \mathcal{O}_{\pi^{(m)}}(-1), 
\pi^{(m)*} \pi^{(m)}_{*} \mathcal{F} 
\right)$ 
contains 
$H^1(A, (L^{\frac{1}{m}})^l) 
\cong H^1(A, \mathcal{O}_{A})$ 
as a direct summand. 
Suppose for a contradiction 
that the sequence (\ref{ses}) 
does not split. 
This is only possible if 
the class 
$\xi \in \Ext^1\left(
\pi^{(m)*}\mathcal{G} \otimes \mathcal{O}_{\pi^{(m)}}(-1), 
\pi^{(m)*} \pi^{(m)}_{*} \mathcal{F} 
\right)$ 
corresponding to the extension (\ref{ses}) 
has the non-zero component 
\[
0 \neq \xi_0 \in H^1(A, \mathcal{O}_A) \subset 
\Ext^1\left(
\pi^{(m)*}\mathcal{G} \otimes \mathcal{O}_{\pi^{(m)}}(-1), 
\pi^{(m)*} \pi^{(m)}_{*} \mathcal{F} 
\right). 
\]

Consider the following Cartesian diagram: 
\[
\xymatrix{
&\mathbb{P}^1 \times A \ar_{u}[d] \ar^{\underline{m}^2 \times \id_{A}}[rr] & &\mathbb{P}^1 \times A \ar_{v}[d] \ar^{\pi_{0}}[r] &A \ar_{\underline{l}}[d]\\
&X^{(\frac{1}{m})} \ar_{g_{m}}[rr] & &X^{(m)} \ar_{\pi^{(m)}}[r] & A.
} 
\]
Since the morphisms $u$ and $v$ 
are flat, we have an isomorphism 
\begin{align*}
v^{*}g_{m*}\mathcal{O}_{\pi^{(\frac{1}{m})}}(a) 
&\cong
(\underline{m}^2 \times \id_{A})_{*}u^*\mathcal{O}_{\pi^{(\frac{1}{m})}}(a) \\
&\cong (\underline{m}^2 \times \id_{A})_{*}\mathcal{O}_{\pi_{0}}(a) 
\end{align*}
and hence it is a direct sum of line bundles 
by the usual toric Frobenius splitting on 
$\mathbb{P}^1$. 
This means that the class $\xi$ 
is mapped to $0$ via the morphism 
$v^* \colon \Ext \to \Ext(v)$, 
where we define 
\begin{align*}
&\Ext:=\Ext^1_{X^{(m)}}\left(
\pi^{(m)*}\mathcal{G} \otimes \mathcal{O}_{\pi^{(m)}}(-1), 
\pi^{(m)*} \pi^{(m)}_{*} \mathcal{F} 
\right), \\
&\Ext(v):=\Ext^1_{\mathbb{P}^1 \times A}\left(
v^*\left( \pi^{(m)*}\mathcal{G} \otimes \mathcal{O}_{\pi^{(m)}}(-1) \right), 
v^*\left( \pi^{(m)*} \pi^{(m)}_{*} \mathcal{F} \right) 
\right). 
\end{align*}

On the other hand, 
$\Ext(v)$ has the direct summand 
$H^1(A, \underline{l}^*\mathcal{O}_A)$ and 
we have the commutative diagram 
\[
\xymatrix{
&\Ext \ar[r]^{v^*} \ar@{->>}[d]_{p} & \Ext(v) \ar@{->>}[d]^{q} \\
&H^1(A, \mathcal{O}_A) \ar[r]_{\underline{l}^*} 
&H^1(A, \underline{l}^*\mathcal{O}_A), 
}
\]
where the vertical arrows are 
the projections to the direct summands, 
and the bottom morphism 
$\underline{l}^* \colon 
H^1(A, \mathcal{O}_{A}) 
\to H^1(A, \underline{l}^*\mathcal{O}_{A})$ 
is an isomorphism. 
In particular, we have 
$0=q \circ v^*(\xi)
=\underline{l}^* \circ p(\xi)
=\underline{l}^*(\xi_0) \neq 0$, 
a contradiction. 

(2) Let 
$X=\mathbb{P}_{C}(\mathcal{O}_{C} \oplus L_{1} \oplus L_{2})$ 
be as in Theorem \ref{cp} (7). 
As in (1), we may assume 
$M=\mathcal{O}_{\pi^{(\frac{1}{m})}}(a)$ and 
$0 \leq a < m^2$. 
Let $\mathcal{F}:=g_{m*}M$. 
As similar to the case (1), 
we have the following exact sequences: 
\begin{equation}
\label{ses2}
\begin{aligned}
& 0 \to \pi^{(m)*}\pi^{(m)}_{*}\mathcal{F} 
\xrightarrow{\alpha} \mathcal{F} 
\to \mathcal{F}' \otimes \mathcal{O}_{\pi^{(m)}}(-1)
\to 0, \\
& 0 \to \pi^{(m)*}\pi^{(m)}_{*}\mathcal{F}' 
\xrightarrow{\beta} \mathcal{F}' 
\to \pi^{(m)*}\mathcal{G} \otimes \mathcal{O}_{\pi^{(m)}}(-1) 
\to 0, 
\end{aligned}
\end{equation}
which correspond to the toric Frobenius splitting 
\[
\underline{m}^2_{*}\mathcal{O}_{\mathbb{P}^2}(a) 
\cong 
\mathcal{O}_{\mathbb{P}^2}^{\oplus k_{0}} 
\oplus \mathcal{O}_{\mathbb{P}^2}(-1)^{\oplus k_{1}} 
\oplus \mathcal{O}_{\mathbb{P}^2}(-2)^{\oplus k_{2}}
\]
on the fiber of $\pi^{(m)}$. 
Let us explain the construction of the exact sequences 
(\ref{ses2}) in detail. 
We can first show that 
the morphism $\alpha$ is injective  as before. 
We set 
$\mcF':=\Coker(\alpha) \otimes \mcO_{\pi^{(m)}}(1)$. 
By restricting to the fiber, we see that 
$\mathcal{F}'|_{\mathbb{P}^2} \cong 
\mathcal{O}_{\mathbb{P}^2}^{\oplus k_{1}} 
\oplus \mathcal{O}_{\mathbb{P}^2}(-1)^{\oplus k_{2}}$, 
and thus the adjoint map $\beta$ is again injective. 
By construction, its cokernel 
$\Coker(\beta)$ 
is semi-orthogonal to the categories 
$\dL\pi^{(m)*}D^b(\bP^2) \otimes \mcO_{\pi^{(m)}}(k)$ 
for $k=0, 1$. 
Hence there exists a sheaf $\mcG$ such that 
$\Coker(\beta) \cong 
\pi^{(m)*}\mathcal{G} \otimes \mathcal{O}_{\pi^{(m)}}(-1)$. 

For $k=1, 2$, let us apply the functor 
$\pi^{(m)}_*\left(- \otimes \mcO_{\pi^{(m)}}(k) \right)$ 
to the first exact sequence in (\ref{ses2}). 
Then we get the exact sequence 
\begin{align*}
0 \to 
\Sym^a&\left(\mcO_C \oplus L_1^{\frac{1}{m}} \oplus L_2^{\frac{1}{m}} \right) 
\otimes \Sym^k\left(\mcO_C \oplus L_1^m \oplus L_2^m \right) \\
&\to 
\Sym^{a+km^2}\left(\mcO_C \oplus L_1^{\frac{1}{m}} \oplus L_2^{\frac{1}{m}} \right)  
\to \pi^{(m)}_*\left(\mcF' \otimes \mcO_{\pi^{(m)}}(k-1) \right) \to 0, 
\end{align*}
which shows that 
the vector bundle 
$\pi^{(m)}_*\left(\mcF' \otimes \mcO_{\pi^{(m)}}(k-1) \right)$
splits into a direct sum of line bundles 
$L_1^{\frac{j_1}{m}} \otimes L_2^{\frac{j_2}{m}}$ 
with $a+1 \leq j_1, j_2 \leq m^2-1$. 
Hence applying the functor 
$\pi^{(m)}_*\left(- \otimes \mcO_{\pi^{(m)}}(1) \right)$ 
to the second exact sequence in (\ref{ses2}), 
we see that the bundle 
$\mcG$ is also a direct sum of line bundles. 

It remains to show that 
the exact sequences in (\ref{ses2}) split. 
Assume that the Ext-groups 
\begin{equation} \label{eq:extP2bdl}
\Ext^1\left( 
\mcF' \otimes \mcO_{\pi^{(m)}}(-1), 
\pi^{(m)*}\pi^{(m)}_*\mcF 
\right), \quad 
\Ext^1\left(
\pi^{(m)*}\mcG \otimes \mcO_{\pi^{(m)}}(-1), 
\pi^{(m)*}\pi^{(m)}_*\mcF'
\right) 
\end{equation}
do not vanish, which is possible only when 
$L_1^{\frac{l_1}{m}} \cong L_{2}^{\frac{l_2}{m}}$ 
for some integers $l_1, l_2 \in \mathbb{Z}$. 
By pulling back the $\bP^2$-bundles 
$X, X^{(m)}$, and $X^{(\frac{1}{m})}$ 
via the multiplication map 
$\underline{l_1l_2} \colon C \to C$, 
the problem is reduced to the case when 
$X=\bP_C\left(\mcO_C \oplus L^a \oplus L^b \right)$ 
for some line bundle $L \in \Pic^0(C)$ 
and integers $a, b \in \mathbb{Z}$. 
Now the groups (\ref{eq:extP2bdl}) 
do not vanish only when the line bundle 
$L \in \Pic^0(C)$ is $l$-torsion 
for some integer $l \in \mathbb{Z}$. 
Again by pulling back the bundles 
via the multiplication map 
$\underline{l} \colon C \to C$,
the situation is further reduced 
to the case when $X=\bP^2 \times C$, 
on which the usual toric Frobenius splitting on $\bP^2$ 
proves the sequences (\ref{ses2}) split. 

(3) Let $X=\mathbb{P}_{C}(\mathcal{O}_{C} \oplus L_{1}) 
\times_{C} \mathbb{P}_{C}(\mathcal{O}_{C} \oplus L_{2})$ 
be as in Theorem \ref{cp} (8), 
let $Y_{i}:=\mathbb{P}_{C}(\mathcal{O}_{C} \oplus L_{i})$. 
Then the problem is reduced to showing 
the corresponding statement for $Y^{(\frac{1}{m})} \to Y^{(m)}$. 
The latter follows from the argument as in (1). 
\end{proof}

We also need the following lemma. 

\begin{lem}
\label{ns}
Let $X$ be as in Theorem \ref{cp} 
(6), (7), or (8) 
which is not necessarily split. 
Then by identifying the tautological classes, 
we have the canonical isomorphism $\Phi$ between 
the N{\' e}ron-Severi group of $X$ 
and that of 
$X_0:=\mathbb{P}^1 \times A$, 
$\mathbb{P}^2 \times C$, 
or $\mathbb{P}^1 \times \mathbb{P}^1 \times C$. 
Furthermore, the following statements hold: 
\begin{enumerate}
\item Under the isomorphism $\Phi$, 
their nef cones are preserved. 
\item Under the isomorphism $\Phi$, 
their classes of the canonical divisors are preserved. 
\item If $X$ is split, then 
the isomorphism $\Phi$ is compatible 
with the formula given in Proposition \ref{toric split} 
in the following sense: 
let $M$ (resp. $M_0$) 
be a line bundle on $X$ (resp. $X_0$) 
such that 
$\Phi(\ch_1(M))=\ch_1(M_0)$. 
Then $\Phi$ induces a bijection between sets 
\[
\left\{ \ch_1(M_j) : 
M_j \mbox{ is a direct summand of } g_{m*}M \mbox{ (resp. } g_{m*}M_0 \mbox{)}
\right\}. 
\]
\end{enumerate}
\end{lem}

\begin{proof}
Let $\pi \colon X=\mathbb{P}_{A}(\mathcal{E}) \to A$ 
be as in Theorem \ref{cp} (6), 
where $\mathcal{E}$ is a rank two vector bundle 
fitting into an exact sequence 
\[
0 \to \mathcal{O}_{A} \to \mathcal{E} \to L \to 0.
\]
We only treat this case. 
We have 
$\NS(X)=\mathbb{Z}[h] \oplus \NS(A)$, 
where $h$ is a divisor such that 
$\mathcal{O}(h)=\mathcal{O}_{\pi}(1)$. 
Hence by identifying a class $[h]$, 
$\NS(X)$ is isomorphic to 
$\NS(\mathbb{P}^1 \times A)$. 

(1) We claim that the line bundle 
$M=\mathcal{O}_{\pi}(a) \otimes \pi^*N$ 
on $X$ is nef if and only if 
$a \geq 0$ and $N$ is nef. 
The `if' direction is clear 
since $\mathcal{O}_{\pi}(1)$ is nef.  
Let us prove the converse. 
Let $h \in |\mathcal{O}_{\pi}(1)|$ 
be a section of $\pi$ 
and $f \cong \mathbb{P}^1$ 
be a fiber of $\pi$. 
Then we have 
$M|_{h} \cong N$, 
$M|_{f} \cong \mathcal{O}_{\mathbb{P}^1}(a)$ 
and they are nef, 
which proves the claim. 
This description of the nef cone 
is independent on the choice of $L \in \Pic^0(A)$ 
and on the choice of the extension class 
$[\mathcal{E}] \in \Ext^1(L, \mathcal{O}_{A})$. 

(2) The canonical line bundle on $X$ is 
given as 
$\mathcal{O}(K_{X})=\mathcal{O}(-2h) \otimes \pi^*L$. 
Since $L \in \Pic^0(A)$, 
we have 
$[K_{X}]=-2[h] \in \NS(X)$ 
in the N{\' e}ron-Severi group 
which is independent on 
the choice of $L \in \Pic^0(A)$. 

(3) The statement is trivial 
from the proof of Proposition \ref{toric split}, 
again by noting that 
$\ch_{1}(L)=0$ for $L \in \Pic^0(A)$. 
\end{proof}

Now we can prove our main theorem: 

\begin{proof}[Proof of Theorem \ref{first thm}]
We only give an outline of the proof 
since the argument is same as \cite{kos17}. 
Let $X$ be as in Theorem \ref{first thm}. 
By Proposition \ref{degenerate}, 
we may assume $X$ is split. 
Take a $\bar{\beta}$-stable object $E$ 
and let $\overline{B}:=B+\bar{\beta}(E)\omega$. 

First assume that $\overline{B}$ is a $\mathbb{Q}$-divisor. 
Take an integer $q \in \mathbb{Z}_{>0}$ 
and an integral divisor $D$ such that 
$\overline{B}=\frac{1}{q}D$. 
For each integer $m \in \mathbb{Z}_{>0}$, 
let us consider the morphism $F_{mq}$ 
constructed in Proposition \ref{endo}. 
Let $D^{(\frac{1}{mq})}$ 
be the divisor on $X^{(\frac{1}{mq})}$ 
such that 
$D^{(\frac{1}{mq})}=\Phi^{-1}(D)$ 
in the cohomology. 
Then the Riemann-Roch theorem 
and Lemma \ref{chern comp} 
imply the inequality 
\begin{align*}
m^6q^6\ch^{\overline{B}}_{3}(E)+\mathcal{O}(m^4)
&=\chi\left(
\mathcal{O}, 
F_{mq}^*E \otimes \mathcal{O}\left(-m^2qD^{(\frac{1}{mq})}
\right)
\right) \\
&\leq 
\hom\left(
\mathcal{O}, 
F_{mq}^*E \otimes \mathcal{O}\left(-m^2qD^{(\frac{1}{mq})}
\right)
\right) \\
& \quad +\ext^2\left(
\mathcal{O}, 
F_{mq}^*E \otimes \mathcal{O}\left(
-m^2qD^{(\frac{1}{mq})}
\right)
\right). 
\end{align*} 

We need to prove that the right hand side 
of the above inequality is of order $m^4$. 
By Proposition \ref{approx ext}, 
to prove 
$\hom\left(
\mathcal{O}, 
F_{mq}^*E \otimes \mathcal{O}\left(-m^2qD^{(\frac{1}{mq})}
\right)
\right) =\mathcal{O}(m^4)$, 
it is enough to find 
an ample divisor $H$ such that 
\[
\Hom\left(
\mathcal{O}(H), 
F_{mq}^*E \otimes \mathcal{O}\left(-m^2qD^{(\frac{1}{mq})}
\right)
\right)=0. 
\]

By using the Serre duality 
and the projection formula, 
we have an isomorphism 
\begin{align*}
&\quad \Hom\left(
\mathcal{O}(H), 
F_{mq}^*E \otimes \mathcal{O}\left(-m^2qD^{(\frac{1}{mq})}
\right)
\right) \\
&\cong 
\Hom\left(
\mathcal{O}(-K_{X^{(m)}}) 
\otimes g_{mq*}\mathcal{O}\left(
H+m^2qD^{(\frac{1}{mq})}+K_{X^{(\frac{1}{m})}}
\right), 
h_{mq}^*E
\right). 
\end{align*}
By Propositon \ref{toric split}, 
the vector bundle 
$\mathcal{O}(-K_{X^{(m)}}) 
\otimes g_{mq*}\mathcal{O}\left(
H+m^2qD^{(\frac{1}{mq})}+K_{X^{(\frac{1}{m})}}
\right)$ 
splits into a direct sum of line bundles $M_{j}$. 
Hence it is enough to show the vanishing 
$\Hom(M_{j}, h_{mq}^*E)=0$ 
for all $j$. 
Since we know the tilt stability of $M_{j}$ 
(resp. $h_{m}^*E$) 
by Lemma \ref{tilt stab of lb} 
(resp. \cite[Proposition 6.1]{bms16}), 
it is enough to show the inequality 
$\nu_{0, h_{mq}^*\overline{B}}(M_{j}) 
>\nu_{0, h_{mq}^*\overline{B}}(h_{mq}^*E)=0$ 
and that the line bundles $M_j$ 
(not $M_j[1]$) 
are in the heart 
$\Coh^{ h_{mq}^*\overline{B}}\left(X^{(mq)} \right)$. 
Both of the requirements are satisfied 
if we can show that 
$\ch_{1}^{h_{mq}^*\overline{B}}(M_{j})$ 
is ample (cf. \cite[Lemma 4.2]{kos17}). 
By Lemma \ref{ns}, 
the problem is now reduced to the case 
when $X$ is $\mathbb{P}^1 \times A$, 
$\mathbb{P}^2 \times C$, 
or $\mathbb{P}^1 \times \mathbb{P}^1 \times C$, 
which is treated in \cite[Lemma 4.6]{kos17}. 
The estimate of $\ext^2$ will also be reduced to 
\cite[Lemma 4.7]{kos17}. 

When $\overline{B}$ is not a $\mathbb{Q}$-divisor 
but an $\mathbb{R}$-divisor, 
we can argue as in \cite[Subsection 4.3]{kos17} 
by using Dirichlet approximation theorem. 
\end{proof}

\section{Proof of Theorem \ref{second thm}}
\label{remaining}
In this section, we will treat 
$\pi \colon 
X:=\mathbb{P}_{\mathbb{P}^2}(\mathcal{T}_{\mathbb{P}^2}) 
\to \mathbb{P}^2$. 
Recall that $X$ is isomorphic to 
a $(1, 1)$-divisor in 
$\mathbb{P}^2 \times \mathbb{P}^2$ 
and hence has two projections to 
$\mathbb{P}^2$: 
\begin{equation}
\label{symm}
\xymatrix{
& &X \ar[ld]_{\pi} \ar[rd]^{\sigma}& \\
&\mathbb{P}^2 & &\mathbb{P}^2. 
}
\end{equation}
Let $h_{1}, h_{2}$ be nef divisors on $X$ 
such that 
$\mathcal{O}(h_{1})=\pi^*\mathcal{O}_{\mathbb{P}^2}(1)$, 
$\mathcal{O}(h_{2})=\sigma^*\mathcal{O}_{\mathbb{P}^2}(1)$. 
Then any line bundle on $X$ can be written as 
$\mathcal{O}(a, b):=
\mathcal{O}(ah_{1}) \otimes \mathcal{O}(bh_{2})$ 
with 
$a, b \in \mathbb{Z}$. 
In this notation, we have 
$\mathcal{O}_{\pi}(1)=\mathcal{O}(1, 1)$. 

Fix an ample divisor $H=ah_{1}+bh_{2}$ 
with $a, b \in \mathbb{Z}_{>0}$. 
For a positive real number $\alpha >0$, 
let $\omega:=\alpha H$. 
We will mainly consider the following 
central charge and heart: 
\[
Z_{\alpha, 0, s}:=
-\ch_{3}+s\alpha^2H^2.\ch_{1}
+i\left(
\alpha H.\ch_{2}-\frac{1}{6}\alpha^3H^3\ch_{0}
\right), 
\]
and $\mathcal{A}_{\alpha, 0}:=\mathcal{A}_{\omega, 0}$. 

First recall the following result 
due to \cite{bms16, bmt14a}. 

\begin{thm}
Fix a positive real number $\alpha >0$. 
Conjecture \ref{original} holds 
for $(X, \alpha H, B=0)$ if and only if 
for every $s > \frac{1}{18}$, 
the pair $(Z_{\alpha, 0, s}, \mathcal{A}_{\alpha, 0})$ 
is a stability condition on $D^b(X)$. 
\end{thm}
\begin{proof}
By \cite[Corollary 5.2.4]{bmt14a}, 
the pair $(Z_{\alpha, 0, s}, \mathcal{A}_{\alpha})$ 
is a stability condition for every $s > \frac{1}{18}$ 
if and only if for every $\nu_{\alpha, 0}$-stable object 
$E \in \Coh^{\alpha H, 0}(X)$ with 
$\nu_{\alpha, 0}(E)=0$, we have 
$\ch_{3} \leq \frac{1}{18}\alpha^2H^2.\ch_{1}(E)$. 
Then the latter is equivalent to 
Conjecture \ref{original} 
by \cite[Theorem 4.2]{bms16}. 
\end{proof}

\begin{defin}
For a fixed ample divisor 
$H=ah_{1}+bh_{2}$, 
we define a real number $\alpha_{0}>0$ as 
\[
\alpha_{0}:=\min\left\{
\sqrt{\frac{1}{a(a+b)}}, 
\sqrt{\frac{18}{a^2+6ab+b^2}}
\right\}. 
\]
\end{defin}

The goal of this subsection 
is to prove the following: 
\begin{prop}
\label{goal}
Let $H=ah_{1}+bh_{2}$ be an ample divisor on $X$ 
with $b>a$. 
Then for every $0 < \alpha < \alpha_{0}$ 
and $s>\frac{1}{18}$, 
the pair $(Z_{\alpha, 0, s}, \mathcal{A}_{\alpha, 0})$ 
is a stability condition on $X$. 
In particular, Conjecture \ref{original} 
holds for $(X, \alpha H, B=0)$. 
\end{prop}

First we prove that the above proposition 
implies Theorem \ref{second thm}. 

\begin{proof}[Proof of Theorem \ref{second thm}]
Let $H=ah_{1}+bh_{2}$ be an ample divisor. 
By the symmetry of the diagram (\ref{symm}), 
we may assume that $b \geq a$. 
Furthermore, if $a=b$, 
then the result is already known 
due to \cite{bmsz17}. 
Now we can assume that $b>a$ 
and then by Theorem \ref{li red}, 
Proposition \ref{goal} implies 
Theorem \ref{second thm}. 
\end{proof}

To prove Proposition \ref{goal}, 
we use the following result due to 
the paper \cite{bmt14a}, 
and follow the arguments 
in \cite{mac14, sch14}. 

\begin{prop}[{\cite[Proposition 8.1.1]{bmt14a}}]
\label{aux}
Assume there exists a heart $\mathcal{C}$ 
in $D^b(X)$ 
with the following properties: 
\begin{enumerate}
\item There exist 
$\phi_{0} \in (0, 1)$ and 
$s_{0} \in \mathbb{Q}$ 
such that 
\[
Z_{\alpha, 0, s_{0}}(\mathcal{C}) 
\subset 
\left\{
r\exp(\pi\phi i) : 
r \geq 0, \phi_{0} \leq \phi \leq \phi_{0}+1
\right\}. 
\]
\item $\mathcal{C} \subset 
\left\langle
\mathcal{A}_{\alpha, 0}, \mathcal{A}_{\alpha, 0}[1]
\right\rangle$. 
\item For any $x \in X$, we have 
$\mathcal{O}_{x} \in \mathcal{C}$ and, 
for all non-zero proper subobjects 
$C \subset \mathcal{O}_{x}$ in 
$\mathcal{C}$, 
we have 
$\Im Z_{\alpha, 0, s_{0}}(C)>0$. 
\end{enumerate}
Then for all $s>s_{0}$, 
the pair $(Z_{\alpha, 0, s}, \mathcal{A}_{\alpha, 0})$ 
is a stability condition on $D^b(X)$. 
\end{prop}

Our heart $\mathcal{C}$ is constructed 
by using an {\it Ext-exceptional collection} 
in the sense of \cite[Definition 3.10]{mac07}. 

\begin{defin}
An exceptional collection 
$E_{1}, \cdots, E_{n}$ 
on a triangulated category $\mathcal{D}$ 
is {\it Ext-exceptional} 
if for all $i \neq j$, we have 
$\Ext^{\leq 0}(E_{i}, E_{j})=0$. 
\end{defin}

\begin{lem}[{\cite[Lemma 3.14]{mac07}}]
Let $E_{1}, \cdots, E_{n}$ 
be a full Ext-exceptional collection 
on a triangulated category $\mathcal{D}$. 
Then the extension closure 
$\langle
E_{1}, \cdots, E_{n}
\rangle_{ex}$ 
is the heart of a bounded t-structure 
on $\mathcal{D}$. 
\end{lem}

\begin{lem}
A collection 
\begin{equation}
\label{ext excep}
\mathcal{O}(-1, -1)[3], 
\mathcal{O}(0, -1)[2], 
\mathcal{O}(1, -1)[1], 
\mathcal{O}(-1, 0)[2], 
\mathcal{O}[1], 
\mathcal{O}(1, 0)
\end{equation}
is a full Ext-exceptional collection on $D^b(X)$. 
\end{lem}
\begin{proof}
Using the equality 
$\mathcal{O}_{\pi}(1)=\mathcal{O}(1, 1)$, 
the collection (\ref{ext excep}) 
can be also written as 
\begin{align*}
&\pi^*\mathcal{O}_{\mathbb{P}^2} \otimes \mathcal{O}_{\pi}(-1)[3], 
\pi^*\mathcal{O}_{\mathbb{P}^2}(1) \otimes \mathcal{O}_{\pi}(-1)[2], 
\pi^*\mathcal{O}_{\mathbb{P}^2}(2) \otimes \mathcal{O}_{\pi}(-1)[1], \\
&\pi^*\mathcal{O}_{\mathbb{P}^2}(-1)[2], 
\pi^*\mathcal{O}_{\mathbb{P}^2}[1], 
\pi^*\mathcal{O}_{\mathbb{P}^2}(1).  
\end{align*}

Since we have 
$D^b(X)
=\left\langle
\dL \pi^*D^b(\mathbb{P}^2) \otimes \mathcal{O}_{\pi}(-1), 
\dL \pi^*D^b(\mathbb{P}^2)
\right\rangle, 
$
we can see that the collection 
(\ref{ext excep}) 
is a full exceptional collection. 
To prove it is Ext-exceptional, 
we can use the formula 
\[
\dR\Gamma\left(X, 
\pi^*\mathcal{O}_{\mathbb{P}^2}(k) \otimes \mathcal{O}_{\pi}(l) 
\right)
=\begin{cases}
0 & (l=-1) \\
\dR\Gamma(\mathbb{P}^2, \mathcal{O}(k)) & (l=0) \\
\dR\Gamma(\mathbb{P}^2, \mathcal{T}_{\mathbb{P}^2}(k)) & (l=1). 
\end{cases}
\]

\end{proof}

Now we can define the following heart: 

\begin{defin}
We define a heart $\mathcal{C} \subset D^b(X)$ as 
\[
\mathcal{C}:=
\left\langle
\mathcal{O}(-1, -1)[3], 
\mathcal{O}(0, -1)[2], 
\mathcal{O}(1, -1)[1], 
\mathcal{O}(-1, 0)[2], 
\mathcal{O}[1], 
\mathcal{O}(1, 0)
\right\rangle_{ex}. 
\]
\end{defin}

The following will be useful 
in the rest of the arguments: 

\begin{lem}
\label{ch comp}
For integers $k, l \in \mathbb{Z}$, 
we have the following equations. 
\begin{enumerate}
\item $H^2.\ch_{1}(\mathcal{O}(k, l))=
la^2+2(k+l)ab+kb^2$. 
\item $H.\ch_{2}(\mathcal{O}(k, l))=
\frac{1}{2}\left(
(2k+l)la+(k+2l)kb \right)$. 
\item $\ch_{3}(\mathcal{O}(k, l))=
\frac{1}{2}kl(k+l)$. 
\end{enumerate}
\end{lem}
\begin{proof}
By using the equations 
$h_{1}^3=h_{2}^3=0$ and 
$h_{1}^2.h_{2}=h_{1}.h_{2}^2=1$, 
the straightforward computation yields the result. 
\end{proof}

\begin{lem}
\label{hearts}
For $0 < \alpha <\alpha_{0}$, 
we have 
$\mathcal{C} \subset 
\left\langle
\mathcal{A}_{\alpha, 0}, 
\mathcal{A}_{\alpha, 0}[1]
\right\rangle_{ex}$. 
\end{lem}
\begin{proof}
By Lemma \ref{ch comp}, 
we have $H.\ch_{1}(\mathcal{O}(1, 0))>0$ 
and hence $\mathcal{O}(1, 0) \in \Coh^{\alpha H, 0}(X)$. 
By assumption on $\alpha$, 
we also have 
\[
H.\ch_{2}(\mathcal{O}(1, 0))
-\frac{1}{6}\alpha^2H^3\ch_{0}(\mathcal{O}(1, 0))
=\frac{1}{2}b-\frac{1}{2}\alpha^2ab(a+b)>0, 
\]
i.e., 
$\nu_{\alpha, 0}(\mathcal{O}(1, 0))>0$. 
Since $\mathcal{O}(1, 0)$ is tilt stable 
by Lemma \ref{tilt stab of lb}, we conclude that 
$\mathcal{O}(1, 0) \in \mathcal{A}_{\alpha, 0}$. 

Similar computations yield that 
\[
\mathcal{O}[1], 
\mathcal{O}(-1, 0)[1], 
\mathcal{O}(1, -1), 
\mathcal{O}(0, -1)[1], 
\mathcal{O}(-1, -1)[1] 
\in \Coh^{\alpha H, 0}(X)
\]
and 
\[
\mathcal{O}[1], 
\mathcal{O}(-1, 0)[2], 
\mathcal{O}(1, -1)[1], 
\mathcal{O}(0, -1)[2], 
\mathcal{O}(-1, -1)[2] 
\in \mathcal{A}_{\alpha, 0}. 
\]
\end{proof}

\begin{lem}
\label{charge}
Let $0 < \alpha < \alpha_{0}$, 
and let $\phi_{0} \in (0, 1)$ be a real number 
such that 
$Z_{\alpha, 0, \frac{1}{18}}(\mathcal{O}(1, 0))=r_{0}\exp(\pi\phi_{0}i)$ 
for some positive real number $r_{0}>0$. 
Then we have 
\[
Z_{\alpha, 0, \frac{1}{18}}(\mathcal{C}) \subset 
\left\{
r\exp(\pi\phi i) : 
r \geq 0, \phi_{0} \leq \phi \leq \phi_{0}+1
\right\}. 
\]
\end{lem}
\begin{proof}
Recall that our central charge 
is written as 
\[
Z_{\alpha}:=Z_{\alpha, 0, \frac{1}{18}}
=-\ch_{3}+\frac{1}{18}\alpha^2H^2.\ch_{1}
+i\left(
\alpha H.\ch_{2}-\frac{1}{6}\alpha^3H^3\ch_{0}
\right). 
\]
By Lemma \ref{ch comp} and the proof of Lemma \ref{hearts}, 
we can see that 
$Z_{\alpha}(\mathcal{O}(-1, -1)[3])$ is 
in the third quadrant, 
$Z_{\alpha}(\mathcal{O}(1, 0))$ is 
in the first quadrant, and 
$Z_{\alpha}(M)$ 
is in the second quadrant 
for other generators $M$ of the heart 
$\mathcal{C}$. 
Now it is enough to check the inequality 
\[
-\frac{\Re Z_{\alpha}\left(\mathcal{O}(1, 0)
\right)}
{\Im Z_{\alpha}\left(\mathcal{O}(1, 0)
\right)}
+\frac{\Re Z_{\alpha}\left(\mathcal{O}(-1, -1)[3]
\right)}
{\Im Z_{\alpha}\left(\mathcal{O}(-1, -1)[3]
\right)}
>0. 
\]
We can estimate 
the left hand side of the above requiered inequality 
as follows: 
\[
\begin{aligned}
&\quad -\frac{\frac{1}{18}\alpha^2(2a+b)b}{\alpha(b-\frac{1}{2}\alpha^2ab(a+b))}
+\frac{1-\frac{1}{18}\alpha^2(a^2+4ab+b^2)}{\alpha(3a+3b-\frac{1}{2}\alpha^2ab(a+b))} \\
&>\frac{-\frac{1}{18}\alpha^2(2a+b)b+1-\frac{1}{18}\alpha^2(a^2+4ab+b^2)}
        {\alpha(3a+3b-\frac{1}{2}\alpha^2ab(a+b))} \\
&=\frac{1-\frac{1}{18}\alpha^2(a^2+6ab+2b^2)}
        {\alpha(3a+3b-\frac{1}{2}\alpha^2ab(a+b))}>0. 
\end{aligned}
\]
Hence the statement holds. 
\end{proof}

\begin{lem}
\label{st of str}
Let $0 < \alpha <\alpha_{0}$ and $x \in X$. 
Then we have $\mathcal{O}_{x} \in \mathcal{C}$. 
Moreover, for any non-zero proper subobject 
$C \subset \mathcal{O}_{x}$ in the category 
$\mathcal{C}$, 
we have 
$\Im Z_{\alpha, 0, \frac{1}{18}}(C)>0$. 
\end{lem}
\begin{proof}
Consider the subcategories 
\begin{align*}
&\mathcal{C}_{1}:=\pi^*\left\langle
\mathcal{O}_{\mathbb{P}^2}(-1)[2], 
\mathcal{O}_{\mathbb{P}^2}[1], 
\mathcal{O}_{\mathbb{P}^2}(1)
\right\rangle_{ex}, \\
&\mathcal{C}_{2}:=\pi^*\left\langle
\mathcal{O}_{\mathbb{P}^2}[2], 
\mathcal{O}_{\mathbb{P}^2}(1)[1], 
\mathcal{O}_{\mathbb{P}^2}(2)
\right\rangle_{ex}
\otimes \mathcal{O}_{\pi}(-1)[1]
\end{align*}
of $\mathcal{C}$. 
Both of these subcategories 
$\mathcal{C}_{i}$ are equivalent 
to the category $\rep(Q, I)$ 
of $Q$-representations with 
certain relations $I$. 
Here $Q$ is the following quiver: 
\begin{equation} \label{eq:quiverP2}
\xymatrix{
&0 \ar@/^/[r] \ar[r] \ar@/_/[r] 
&1 \ar@/^/[r] \ar[r] \ar@/_/[r] 
&2
}
\end{equation}

Let $y:=\pi(x)$ 
and denote $L_{y}:=\pi^{-1}(y) \cong \mathbb{P}^1$. 
Then we have the following exact triangle 
in $D^b(X)$ 
\[
\mathcal{O}_{L_{y}} \to \mathcal{O}_{x} \to \mathcal{O}_{L_{y}}(-1)[1]
\]
with $\mathcal{O}_{L_{y}} \in \mathcal{C}_{1}$ 
and $\mathcal{O}_{L_{y}}(-1)[1] \in \mathcal{C}_{2}$. 
This proves that $\mathcal{O}_{x} \in \mathcal{C}$. 
By Lemma \ref{lem:repP2} below, 
the $Q$-representations corresponding to 
$\mathcal{O}_{L_{y}} \in \mathcal{C}_{1}$ 
and 
$\mathcal{O}_{L_{y}}(-1)[1] \in \mathcal{C}_{2}$ 
are the same representation, 
which has dimension vector $(1, 2, 1)$ 
and is generated by the vertex $0$. 
We say that $\mathcal{O}_{x}$ 
has dimension vector $(1, 2, 1, 1, 2, 1)$. 

To prove the second statement, 
recall that for an object $M$ in (\ref{ext excep}), 
we have 
$\Im Z_{\alpha, 0, \frac{1}{18}}(M)<0$ 
if and only if 
$M=\mathcal{O}(-1, -1)[3]
=\pi^*\mathcal{O}_{\mathbb{P}^2} \otimes \mathcal{O}_{\pi}(-1)[3]$. 
Hence it is enough to consider 
a subobject $C \subset \mathcal{O}_{x}$ 
with dimension vector 
$(1, a, b, c, d, e)$. 
We must prove that $C=\mathcal{O}_{x}$ 
for such a subobject $C$.  
There exists an exact sequence 
\[
0 \to T_{1} \to C \to T_{2} \to 0
\]
in $\mathcal{C}$ 
with some objects $T_{i} \in \mathcal{C}_{i}$. 
Using the definition of the Ext-exceptional collection, 
we can see that 
$T_{1} \subset \mathcal{O}_{L_{y}}$ 
(resp. $T_{2} \subset \mathcal{O}_{L_{y}}(-1)[1]$). 
Since we have assumed that the dimension vector of $C$ 
is $(1, a, b, c, d, e)$, 
and since $\mathcal{O}_{L_{y}}(-1)[1]$ is 
generated by vertex $0$ as a quiver representation, 
we must have $T_{2}=\mathcal{O}_{L_{y}}(-1)[1]$. 
Now we get the commutative diagram 
\[
\xymatrix{
& &0 \ar[d] &0 \ar[d] & & \\
&0 \ar[r] &T_{1} \ar[d] \ar[r] &C \ar[d] \ar[r] &T_{2} \ar@{=}[d] \ar[r] &0 \\
&0 \ar[r] &\mathcal{O}_{L_{y}} \ar[d] \ar[r] 
    &\mathcal{O}_{x} \ar[d] \ar[r] &\mathcal{O}_{L_{y}}(-1)[1] \ar[r] &0 \\
& &K \ar[d] \ar@{=}[r] &K \ar[d] & & \\
& &0 &0 & &
}
\]
for some $K \in \mathcal{C}_{1}$. 
However, since 
$\Hom(\mathcal{O}_{x}, \mathcal{C}_{1})=0$, 
we must have $K=0$, 
i.e., $C=\mathcal{O}_{x}$ 
as required. 
\end{proof}

We have used the following lemma, 
which seems to be well-known: 
\begin{lem} \label{lem:repP2}
For a given integer $i \in \mathbb{Z}$, 
let 
\[
\mathcal{D}_i:=\left\langle 
\mathcal{O}_{\mathbb{P}^2}(i-2)[2], 
\mathcal{O}_{\mathbb{P}^2}(i-1)[1], 
\mathcal{O}_{\mathbb{P}^2}(i) 
\right\rangle_{ex}
\] 
be the heart of a bounded t-structure 
on $D^b(\mathbb{P}^2)$ 
generated by the Ext-exceptional collection. 
The following statements hold: 
\begin{enumerate}
\item We have an equivalence 
$\mathcal{D}_i \cong \rep(Q, I)$ 
of abelian categories, 
where $Q$ is the quiver given in (\ref{eq:quiverP2}) 
and $I$ is certain relations. 

\item For every point $y \in \mathbb{P}^2$, 
the structure sheaf 
$\mathcal{O}_y \in \mathcal{D}_i$ 
has dimension vector $(1, 2, 1)$, 
and is generated by the vertex $0$. 
\end{enumerate}
\end{lem}
\begin{proof}
The first assertion is well-known, 
see \cite{bei78, bon89}. 
Let us prove the second assertion. 
Since we have 
$\mathcal{O}_y \otimes \mathcal{O}_{\mathbb{P}^2}(1) 
\cong \mathcal{O}_y$, 
we may assume $i=0$. 
Note that the objects 
$\mathcal{O}_{\mathbb{P}^2}(-2)[2], 
\mathcal{O}_{\mathbb{P}^2}(-1)[1], 
\mathcal{O}_{\mathbb{P}^2} \in \mathcal{D}_0$ 
correspond to the simple $(Q, I)$-representations 
of dimension vectors 
$(1, 0, 0), (0, 1, 0), (0, 0, 1)$, respectively.  
Denote by $l \subset \mathbb{P}^2$ a line. 
We have the following exact triangles 
\[
\mathcal{O}_l \to \mathcal{O}_y \to \mathcal{O}_l(-1)[1], \quad 
\mathcal{O}_{\mathbb{P}^2}(-j)[j] \to \mathcal{O}_l(-j)[j] 
\to \mathcal{O}_{\mathbb{P}^2}(-j-1)[j+1] 
\]
for $j=0, 1$, 
and hence $\mathcal{O}_y \in \mathcal{D}_0$ 
has dimension vector $(1, 2, 1)$. 

Let us consider a subobject $S \subset \mathcal{O}_y$ 
in the category $\mathcal{D}_0$ with dimension vector 
$(1, s, t)$. 
Since $\Hom(\mathcal{O}_{\mathbb{P}^2}(-j)[j], \mathcal{O}_y)=0$ 
for $j=1, 2$, we must have $t \neq 0$. 
Then the quotient $\mathcal{O}_y/S$ 
has dimension vector $(0, 2-s, 0)$. 
On the other hand, we also have the vanishing 
$\Hom(\mathcal{O}_y, \mathcal{O}_{\mathbb{P}^2}(-1)[1])=0$, 
and hence we must have $\mathcal{O}_y/S=0$, i.e., 
$T=\mathcal{O}_y$ as required. 
\end{proof}

Now we can prove 
Proposition \ref{goal}. 

\begin{proof}[Proof of Proposition \ref{goal}]
By Lemma \ref{hearts}, Lemma \ref{charge}, 
and Lemma \ref{st of str}, 
we can apply Proposition \ref{aux} 
to get the result. 
\end{proof}



\end{document}